\documentclass[11pt,letterpaper]{article}

\usepackage{fullpage}
\usepackage{stylefile}
\usepackage[colorinlistoftodos,bordercolor=orange,backgroundcolor=orange!20,linecolor=orange,textsize=scriptsize]{todonotes}

\title{Coordinate Descent with Arbitrary Sampling I: \\Algorithms and Complexity\thanks{The  authors  acknowledge support from the EPSRC  Grant EP/K02325X/1,
{\em Accelerated Coordinate Descent Methods for Big Data Optimization}. Most of the material of this paper was obtained by the authors in Spring 2014, and was presented by PR  in June 2014 at the  ``Khronos Days Summer School'' focused on
``High-Dimensional Learning and Optimization'' in Grenoble, France \cite{Grenoble_theory}; \url{http://www.maths.ed.ac.uk/\%7Eprichtar/docs/cdm-talk.pdf}.}}

\author{Zheng Qu \footnote{School of Mathematics, The University of Edinburgh, United Kingdom (e-mail: zheng.qu@ed.ac.uk)} \qquad \qquad 
 Peter Richt\'{a}rik \footnote{School of Mathematics, The University of Edinburgh, United Kingdom (e-mail: peter.richtarik@ed.ac.uk) }}

\begin{document}
\maketitle

\begin{abstract} We study the problem of minimizing the sum of a smooth convex function and a convex block-separable regularizer and
propose a new randomized coordinate descent method, which we call ALPHA. Our method at every iteration  updates a {\em random subset}  of coordinates, following an {\em arbitrary distribution}. No coordinate descent methods capable to handle an arbitrary sampling have been studied in the literature before for  this problem. ALPHA is a remarkably flexible algorithm: in special cases, it reduces to  deterministic and randomized methods such as gradient descent, coordinate descent, parallel coordinate descent and distributed coordinate descent -- both in nonaccelerated and accelerated variants. The variants with arbitrary (or importance) sampling are new.  We provide a complexity analysis of ALPHA,  from which we deduce as a direct corollary complexity bounds for its many variants, all matching or improving  best known bounds.
\end{abstract}

\section{Introduction}

With the dawn of the big data age, there has been a growing interest in solving optimization problems of unprecedented sizes. It was soon realized that traditional approaches, which work extremely well for problems of moderate sizes and when solutions of high accuracy are required, are not efficient for modern problems of large enough  size and for applications where only rough or moderate accuracy solutions are sufficient.  The focus of the optimization, numerical analysis and machine learning  communities, and of practitioners in the sciences and industry,  shifted to first-order (gradient) algorithms \cite{NesterovIntro}. 

However, once the size of problems becomes truly big, it is necessary to turn to methods which are able to output a reasonably good solutions after an amount of work roughly equivalent to reading the data describing the problem a few times. For this to be possible, methods need to be able to progress while reading only a small part of the data describing the problem, which often means that a single iteration needs to be based on less information than that contained in the gradient of the objective (loss) function. The most popular methods of this type are stochastic gradient methods \cite{tongSGD, nemirovski2009robust, Pegasos-MAPR, Pegasos2, IProx-SDCA}, randomized coordinate descent methods \cite{PCD2007, SDCA-2008, UCDC, PCDM, Pegasos2, ICD, DQA, SPCDM, Hydra, Hydra2, IProx-SDCA, MinibatchASDCA, APROXSDCA, APCG, QUARTZ, S2CD} and semi-stochastic gradient descent methods \cite{SAG, zhanglijun, SVRG,  S2GD, Mixed2013,  MISO, SAGA, proxSVRG, mS2GD, S2CD}.

\subsection{Randomized coordinate descent}

In this paper we focus on randomized coordinate descent methods. After the seminal work of Nesterov~\cite{Nesterov:2010RCDM}, which provided an early  theoretical justification of these methods for unconstrained convex minimization, the study has been successively extended to $L1$-regularized \cite{Shalev-Shwartz2011,TTD}, proximal~\cite{UCDC,RBCDMLuXiao}, parallel~\cite{PCDM, APPROX}, distributed~\cite{Hydra, Hydra2, QUARTZ} and primal-dual \cite{Pegasos2, QUARTZ} variants of  coordinate descent. Accelerated coordinate decent---characterized by its  $O(1/k^2)$ complexity for non-strongly convex problems---was studied in~\cite{Nesterov:2010RCDM,RBCDMLuXiao}. However, these methods are of theoretical nature only due to the fact that they rely on the need to perform full-dimensional vector operation at every iteration, which destroys the main advantage of coordinate descent -- its ability to reduce the problem into subproblems of smaller sizes. A theoretically and practically efficient accelerated coordinate descent methods were proposed recently by Lee and Sidford~\cite{lee2013efficient} and Fercoq and Richt\'{a}rik~\cite{APPROX}, the latter work (APPROX algorithm) combining acceleration with parallelism and proximal setup. An accelerated distributed coordinate descent algorithm~\cite{Hydra2} is obtained by specializing APPROX to a distributed sampling. All above mentioned papers only consider unconstrained or separably constrained problems. Some progress on linearly-coupled constraints has been made by Necoara et al in~\cite{Necoara:Parallel,Necoara:rcdm-coupled}. Asynchronous variants of parallel coordinate descent methods were developed by Liu, Wright et al~\cite{WrightAsynchrous13,WrightAsynchrous14}.

Virtually all existing work in stochastic optimization deals with a {\em uniform sampling}. In the context of coordinate descent, this means that the random subset (sampling) of coordinates chosen and updated at every iteration has the property that each coordinate is chosen {\em equally likely}.  The possibility to assign different selection probabilities to different coordinates---also known as {\em importance sampling}---was considered in~\cite{Nesterov:2010RCDM, UCDC} and recently in \cite{IProx-SDCA, SPDC}. However, these works consider the serial case only:  a single coordinate is updated in each iteration.  Randomized coordinate descent methods updating a {\em subset of coordinates following an arbitrary distribution} (i.e., using an arbitrary sampling) were first investigated by Richt\'{a}rik and Tak\'{a}\v{c}~\cite{NSync} (NSync method)  for strongly convex and smooth objective functions, and subsequently by Qu, Richt\'{a}rik and Zhang~\cite{QUARTZ} (QUARTZ method),  for strongly convex and possibly nonsmooth functions, and in a primal-dual framework. 

In this paper we give the {\em first fully unified analysis} of gradient type algorithms which contain randomized coordinate descent on one end of the spectrum and gradient (or accelerated gradient) descent on the other hand. All our complexity results match or improve on the state of the art in all cases where specialized algorithms for specific samplings already exist. Moreover, we managed to substantially simplify the analysis for the sake of making the material accessible to a wide community.

\subsection{Problem Formulation}

In this paper we consider the composite  optimization problem 
\begin{equation}\label{eq-probse}
\begin{array}{ll}
 \mathrm{minimize} & F(x)\eqdef f(x)+\psi(x)\\
\mathrm{subject~to~} & x=(x^1,\dots,x^n)\in \R^{N_1}\times\dots\times \R^{N_n}=\R^N,
\end{array}
\end{equation}
where $f:\R^N \rightarrow \R$ is convex and differentiable, $\psi:\R^N\rightarrow \R\cup\{+\infty\}$ is   block-separable:
$$
\psi(x)=\sum_{i=1}^n\psi^i(x^i),
$$
and each $\psi^i:\R^{N_i}\rightarrow \R\cup\{+\infty\}$ is convex and closed. 

\subsection{Contributions}

We now summarize the main contributions of this work.



\paragraph{New algorithm.} We propose ALPHA (Algorithm~\ref{alg:general})  -- a randomized gradient-type method  for  solving  the convex composite  optimization problem \eqref{eq-probse}. In each iteration, ALPHA picks and updates a random subset of the blocks $\{1,2,\dots,n\}$, using an  {\em arbitrary sampling.} That is, we allow for the distribution of the random set-valued mapping to be arbitrary (and as explained further below, analyze the iteration complexity of the method).

To the best of our knowledge, there are only two  methods in the literature with the ``arbitrary sampling'' property, the NSync method of Richt\'{a}rik and Tak\'{a}\v{c} \cite{NSync} (focusing on the simple problem of minimizing a smooth strongly convex function) and the QUARTZ method of Qu, Richt\'{a}rik and Zhang \cite{QUARTZ} (a primal-dual method; considering strongly convex but possibly nonsmooth functions appearing in machine learning). Hence, our work is complementary to this development.

\paragraph{Complexity analysis.} We study the iteration complexity of ALPHA. That is, for an arbitrary (but ``proper'') sampling, we provide bounds on the number of iterations needed to approximately solve the problem, in expectation. Our general bounds  are formulated in Section~\ref{sec-proximalmin}: Theorem~\ref{th-proximal_normal} covers the non-accelerated variant with $O(1/k)$ rate and Theorem~\ref{th-proximal_fast} covers the accelerated variant with $O(1/k^2)$ rate, where $k$ is the iteration counter. To the best of our knowledge, these are the first complexity results for a randomized coordinate descent methods utilizing an {\em arbitrary sampling} for problem  \eqref{eq-probse}.  

\paragraph{Expected separable overapproximation.} Besides the dependence of the complexity bound on the iteration counter $k$, it is important to study its dependence on the sampling $\hat{S}$ and the objective function. Our results make this dependence explicit: they hold under the assumption that $f$ admits an expected separable overapproximation (ESO) with respect to the sampling $\hat{S}$ (Assumption~\ref{ass-eso}).  This is an inequality involving $f$ and $\hat{S}$ which determines certain important parameters $v_1,\dots,v_n$ which are needed to run the method (they determine the stepsizes) and which also appear in the complexity bounds.  In some cases it is possible to {\em design} a sampling which optimizes the complexity bound.

In the case of a {\em serial sampling}, which is by far the most common type of sampling studied in conjunction with randomized coordinate descent methods ($\hat{S}$ is serial if $|\hat{S}|=1$ with probability 1), the parameter $v_i$ can simply be set to the Lipschitz constant of the block-derivative of $f$ corresponding to block $i$. In particular, if $n=1$, then $v_1$ is the Lipschitz constant of the gradient of $f$~\cite{Nesterov:2010RCDM, UCDC}. 
The situation is more complicated in the case of a {\em parallel sampling} ($\hat{S}$ is parallel if it is not serial; that is, if we allow for multiple blocks to be updated at every iteration) -- and this why there is a need for the ESO inequality. Intuitively speaking, the parameters $v_1,\dots,v_n$ capture certain smoothness properties of the gradient of $f$ in a random subspace spanned by the blocks selected by the sampling $\hat{S}$.

The ESO concept is of key importance in the design and analysis of  randomized coordinate descent methods \cite{PCDM, SPCDM, Hydra, NSync,APPROX, Hydra2, QUARTZ, paper2}. We  provide a {\em systematic study  of ESO inequalities} in a companion paper~\cite{paper2}.

\paragraph{Simple complexity analysis in the smooth case.} In order to make the exposition more accessible, we first focus on ALPHA applied to problem \eqref{eq-probse} with $\psi\equiv 0$ (we call this the ``smooth case''). In this simpler setting, it is possible to provide a {\em simplified complexity analysis} -- we do this in Section~\ref{sec-smooth}; see Theorem~\ref{th-smooth_normal} (non-accelerated variant with $O(1/k)$ rate) and Theorem~\ref{th-smooth_fast} (accelerated variant with $O(1/k^2)$ rate). For convenience, ALPHA specialized to the smooth case is formulated as Algorithm~\ref{alg:smooth}.  Our analysis in this case is different from the one we give in Section~\ref{sec-proximalmin}, where we analyze the method in the general proximal setup.  

\paragraph{Flexibility.}  ALPHA is a remarkably flexible\footnote{We have named the method ALPHA because of this flexibility: ``ALPHA'' as a single source from which one obtains diversity.} algorithm, encoding a number of classical, recent and new algorithms in special cases, depending on the choice of the parameters of the method: sampling $\hat S$ and  ``stepsize sequence'' $\{\theta_k\}$. We devote Section~\ref{sec-specialcasesm} to highlighting several of the many algorithms ALPHA reduces to in special cases, focusing on the smooth case for simplicity (special cases in the general proximal setting are discussed in the appendix). In particular, if $\hat{S}=\{1,2,\dots,n\}$ with probability 1, ALPHA reduces to a deterministic method:  {\em gradient descent} (GD; Algorithm~\ref{alg:gd}) or {\em accelerated  gradient descent} (AGD; Algorithm~\ref{alg:accgd}), depending on the choice of the sequence $\{\theta_k\}$.  For a non-deterministic  sampling, we obtain {\em parallel coordinate descent} (PCD; Algorithm~\ref{alg:rcdm}) and {\em  accelerated parallel coordinate descent} (APCD; Algorithm~\ref{alg:accsmooth}) with {\em arbitrary sampling} -- which is new. If a uniform sampling is used, PCD reduces to the  PCDM algorithm \cite{PCDM}. If a distributed
\footnote{Distributed sampling is a structured uniform sampling first introduced in \cite{Hydra} and further studied in \cite{Hydra2,QUARTZ} and in the companion paper \cite{paper2}. In a distributed sampling, the blocks $\{1,2,\dots,n\}$ are first partitioned into $c$ sets of equal cardinality (it is useful to think of $c$ to be equal to the number of compute nodes in a distributed computing environment). The sampling is constructed by letting each node choose a subset of a fixed size (say $\tau$) of the blocks it owns, uniformly at random and independently from others, and then taking the union of these random sets. This union is a random subset of the set of blocks; and is called the $(c,\tau)$-distributed sampling.} 
sampling is used instead, PCD reduces to Hydra \cite{Hydra}. Similarly, if a uniform sampling is used, APCD reduces to APPROX \cite{APPROX} (in fact, our version of APPROX is a bit more flexible with respect to choice of $\theta_0$, which leads to a better complexity result). APCD specialized to a distributed sampling reduces to Hydra$^2$ \cite{Hydra2}.

\paragraph{Robustness.} Since we establish a complexity result for an arbitrary sampling, one of the key contributions of this work is to show  that coordinate descent methods are robust to the choice of the sampling $\hat{S}$. In many applications one is {\em forced} to sample the coordinates/blocks in a  non-traditional way and up to this point the issue of whether the resulting algorithm would converge (let alone the issue of estimating its complexity) was open. For instance, in many metric learning / matrix problems one wishes to find a positive semidefinite matrix satisfying certain properties. It is often efficient to work with an algorithm which would in each iteration update all the elements in a certain row and the corresponding column of the matrix. If we think of the elements of this matrix as coordinates, then any sampling induced in this way puts more probability on the diagonal elements of than on the off-diagonal elements. An algorithm of this type was not analyzed before. The complexity of such a method would follow as a special of our general results specialized to the corresponding sampling.

\paragraph{Improved complexity results.} In all cases where ALPHA reduces to an existing method, our complexity bound either matches the best known bound for that method or improves upon the best known bound. For instance, while the complexity of PCDM \cite{PCDM} (which coincides with PCD specialized to a uniform sampling; Algorithm~\ref{alg:rcdm}) depends on  the size of a certain level-set of $f$ (and in particular, requires the level set to be bounded), our bound  does not involve this quantity (see Section~\ref{subsec-PCDM}). Another example is  APPROX \cite{APPROX} (which is closely related to  APCD specialized to a uniform sampling): we obtain a more compact and improved result.

\paragraph{Two in one.} We provide a {\em unified } complexity analysis covering the nonaccelerated and accelerated variants of ALPHA. This is achieved by establishing a certain key technical recursion (Lemma~\ref{lem:recursion2}) for an {\em arbitrary } choice of the parameters $\{\theta_k\}$. Since the two variants of ALPHA differ in the choice of this sequence only, the analysis of both is identical up to this point. The recursion is then analyzed in two different ways, depending on the sequence $\{\theta_k\}$, which leads to the final complexity result.

\paragraph{Efficient implementation.}  As formulated in Algorithm~\ref{alg:general}, ALPHA  seems to require that two vectors in $\R^N$ be added at each iteration (unless the three sequences coincide, which happens in some important special cases). Motivated by \cite{lee2013efficient,APPROX}, in Section~\ref{sec-eie} we give an equivalent form of Algorithm~\ref{alg:general}, which under some structural assumptions on $f$ (see~\eqref{eq-efff}) does not require such full-dimensional operations. This is important as the efficiency of coordinate descent methods largely stems from their ability to decompose the problem into  subproblems, in an iterative fashion, of much smaller size than is the size of the original problem.

\subsection{Outline of the paper}

The paper is organized as follows. In Section~\ref{sec-algo} we establish  notation, describe the ALPHA algorithm and comment on the key assumption: Expected Separable Overapproximation (ESO). We defer the in-depth study ESO inequalities to a companion paper \cite{paper2}. In Section~\ref{sec-smooth} we give a simple complexity proof of ALPHA in the smooth case ($\psi=0$).
Subsequently, in Section~\ref{sec-specialcasesm} we present four algorithms  that ALPHA reduces to in  special cases, and state the corresponding complexity results, which follow from our general result, in a simplified form.  We do this for the benefit of the reader. In Section~\ref{sec-eie} we provide an equivalent form of writing ALPHA--one leading to an efficient implementation avoiding full dimensional operations.  In Section~\ref{sec-proximalmin} we state and prove the convergence result of ALPHA when applied to the general proximal minimization problem \eqref{eq-probse}. Finally, in Section~\ref{sec:conclusion} we conclude and in the appendix we comment on several special cases ALPHA reduces to in the general proximal setup.
  

\section{The Algorithm}\label{sec-algo}

In this section we first formalize the block structure of $\R^N$ and establish necessary notation (Section~\ref{subsec:alg_ESO}), then proceed to describing the ALPHA algorithm (Section~\ref{sec-desalgo}) and finally comment in the key assumption needed for our complexity results (Section~\ref{subsec:alg_ESO}).

\subsection{Preliminaries} \label{subsec:alg_prelim}

\textbf{Blocks.} We first describe the block setup which has become standard in the analysis of block coordinate descent methods~\cite{Nesterov:2010RCDM, UCDC, PCDM,APPROX}.
The space $\R^N$ is decomposed into $n$
subspaces: $\R^N=\R^{N_1}\times\dots\times \R^{N_n}$. Let $\bU$ be the $N\times N$ identity matrix and $\bU=[\bU_1,\dots,\bU_n]$ be its decomposition into column submatrices 
$\bU_i\in \R^{N\times N_i}$. 
For $x\in \R^N$, let $x^i \in \R^{N_i}$ be the block of coordinates corresponding to the columns of $\bU_i$, i.e., $x^i=\bU_i^\top x$.
For any $h\in \R^N$ and $S\subseteq[n]\eqdef\{1,\dots,n\}$ we define:
\begin{equation}\label{eq:decomposition_of_x}
h_{[S]}\eqdef \sum_{i\in S} \bU_i h^i.
\end{equation}
For function $f$, we denote by $\nabla f(x)$ the gradient of $f$ at point $x\in\R^N$
 and by $\nabla_i f(x)\in \R^ {N_i}$ the block of partial derivatives $\nabla_i f(x)=\bU_i^\top \nabla f(x)$.

\textbf{Norms.} The standard Euclidean inner product (with respect to the standard basis) in spaces $\R^N$ and $\R^{N_i}$, $i\in [n]$, 
will be denoted  by $\ve{\cdot}{\cdot}$. That is, for vectors $x,y$ of equal size, we have $\ve{x}{y} = x^\top y$.  Each space $\R^{N_i}$ is equipped with a Euclidean norm:
\begin{equation}\label{block_norm}\|x^i\|_i^2 \eqdef \ve{\bB_ix^i }{x^ i} = (x^i)^\top \bB_i x^i,
\end{equation}
where $\bB_i$ is an $N_i$-by-$N_i$ positive definite matrix. 
For $w\in \R_{++}^n$ and $x,y\in \R^N$ we further define
\begin{align}\label{a-wproduct}
\ve{x}{y}_{w}\eqdef \sum_{i=1}^n w_i\<x^i,y^i> 
\end{align}
and 
\begin{align}\label{a-norm}
\|x\|_{w}^2\eqdef{\sum_{i=1}^n w_i \|x^ i\|_i^2} \overset{\eqref{block_norm}}{=} \sum_{i=1}^n w_i \< \bB_i x^i, x^i >.
\end{align}
For $x\in \R^N$, by $\bB x$ we mean the vector $\bB x = \sum_{i=1}^n \bU_i \bB_i x^i$. That is, $\bB x$ is the vector in $\R^N$ whose $i$th block is equal to $\bB_i x^i$. For vectors $x,y\in \R^N$ we have
\begin{equation}\label{eq:suih9s8hs}\|x+y\|_w^2 = \|x\|_w^2 + 2\ve{\bB x}{y}_w + \|y\|_w^2.\end{equation}

\textbf{Vectors.} For any two  vectors  $x$ and $y$ of the same size, we denote by $x\circ y$ their Hadamard (i.e., elementwise) product.  By abuse of notation, we denote by $u^2$ the elementwise square of the vector $u$, by $u^{-1}$ the elementwise inverse of 
 vector $u$ and  by $u^{-2}$ the elementwise square of $u^{-1}$.  For vector $v\in \R^n$ and $x\in \R^N$ we will write
\begin{equation}\label{eq:u9s8y9sh0j0909}v\cdot x \eqdef \sum_{i=1}^n v_i \bU_i x^i.\end{equation}
That is, $v\cdot x$ is the vector obtained from $x$ by multiplying its block $i$ by $v_i$ for each $i\in [n]$. If all blocks are of size one ($N_i=1$ for all $i$), then $v\cdot x = \Diag(v)x$ where $\Diag(v)$ is the diagonal matrix with diagonal vector $v$.

\subsection{ALPHA}\label{sec-desalgo}

In this section we describe ALPHA (Algorithm~\ref{alg:general}) -- a randomized block coordinate descent method for solving \eqref{eq-probse}.

We denote by $\dom \psi$ the domain of the proximal term $\psi$.

\begin{algorithm}[H]
\begin{algorithmic}[1]
\STATE \textbf{Parameters:}  proper  sampling $\hat S$ with probability vector $p=(p_1,\dots,p_n)$,   $v \in \R^n_{++}$, sequence $\{\theta_k\}_{k\geq 0}$
\STATE \textbf{Initialization:} {choose $x_0\in \dom \psi$ and set  $z_0=x_0$}
\FOR{$k\geq 0$} 
\STATE $y_k=(1-\theta_k)x_k+\theta_kz_k$ \\
\STATE Generate a random set of blocks $S_k\sim \hat S $\\
\STATE $z_{k+1}\leftarrow z_k$
\FOR{ $i\in S_k$} 
\STATE
$
z_{k+1}^i=\argmin_{z\in \R^{N_i}} \big\{\< \nabla_i f(y_k), z>+\frac{\theta_k  v_i }{2p_i} \|z-z_k^i\|_i^2+\psi^i(z) \big\}
$ \\
\ENDFOR\\
\STATE $x_{k+1}=y_k+\theta_k p^{-1}\cdot  (z_{k+1}-z_k)$
\ENDFOR
\end{algorithmic}
\caption{ALPHA}
\label{alg:general}
\end{algorithm}

To facilitate the presentation, we first recall some basic facts and terminology related to samplings (for a broader coverage see~\cite{PCDM,APPROX,QUARTZ, paper2}. A sampling $\hat{S}$ is a random set valued mapping with values in $2^{[n]}$.  We say that sampling $\hat{S}$ is
  {\em nil} if $\Prob(\hat{S}=\emptyset)=1$, 
{\em proper} if $\Prob(i \in \hat{S})>0$ for all $i \in [n]$, {\em uniform} if  $\Prob(i \in \hat{S})=\Prob(i' \in \hat{S})$ for all $i,i'\in [n]$ and {\em serial} if $\Prob(|\hat S|= 1)=1$.  The probability that the block $i$ is chosen is denoted by:
\begin{equation}
\label{eq:p_i}p_i\eqdef \Prob(i\in \hat{S}), \qquad i\in [n].\end{equation}
It is easy to see that if $\hat S$ is uniform, then $p_i=\Exp[|\hat S|]/{n}$, for $i\in[n]$.

Algorithm~\ref{alg:general} starts from an initial vector $x_0\in \R^ N$ and generates three sequences $\{x_k,y_k,z_k\}_{k\geq 0}$.
At  iteration $k$, a random subset of blocks, $S_k\subseteq[n]$, is generated according to the distribution of sampling $\hat S$ (a parameter to enter
at the beginning of the algorithm). In order to guarantee that each block has a nonzero probability to be selected, it is necessary to assume that $\hat{S}$ is {\em proper}. 

To move from $z_k$ to $z_{k+1}$,
we only need to evaluate $|S_k|$ partial derivatives
 of $f$ at point $y_k$ and update only the blocks of $z_k$ belonging to $S_k$ 
to the solutions of $|S_k|$ proximal problems (Steps 6 to 8). 
The vector $x_{k+1}$ is obtained from $y_k$ by changing only the blocks of $y_k$ belonging to $S_k$ (Step 10). The vector 
$y_k$ is a convex combination of $x_k$ and $z_k$ (Step 4) with coefficient $\theta_k$, a parameter to be chosen between $(0,1]$. 
Note that unless $\theta_k p_i=1$ for all $i \in [n]$ and $k\in \bN$, in which case the three sequences $\{x_k,y_k,z_k\}_{k\geq 0}$ reduce to one same sequence, the update in Step 4 requires a full-dimensional vector operation, as previously remarked in~\cite{Nesterov:2010RCDM, APPROX}.
In Section~\ref{sec-eie} we will provide an equivalent form of 
Algorithm~\ref{alg:general}, which avoids full-dimensional vector operations for special forms of $f$.

Let us extract the relations between the three sequences.
Define
\begin{align}\label{a-tildezk}
\tilde z_{k+1}\eqdef \displaystyle\argmin_{z\in \R^{N}}\{\<\nabla f(y_k),z>+\frac{\theta_k}{2}\|z-z_k\|^2_{p^{-1}\circ v}+ \psi(z)\}.
\end{align}
Then
\begin{align}\label{a-zk21}
z_{k+1}^i=\left\{\begin{array}{ll} \tilde z_{k+1}^i & i\in S_k \\ z_k^i & i\notin S_k\end{array}\right. ,
\end{align}
and hence $z_{k+1}-z_k = (\tilde{z}_{k+1}-z_k)_{[S_k]}$
and
\begin{align}\label{a-xkplus1yk}
 x_{k+1}=y_k+\theta_k p^{-1}\cdot  (\tilde z_{k+1}-z_k)_{[S_k]} .
\end{align}
Note also that from the definition of $y_k$ in Algorithm~\ref{alg:general}, we have:
\begin{align}\label{a-xkykzk}
 \theta_k(y_k-z_k)=(1-\theta_k)(x_k-y_k).
\end{align}

\subsection{Expected separable overapproximation} \label{subsec:alg_ESO}

To guarantee the convergence of Algorithm~\ref{alg:general}, we shall require that $f$ admits an expected separable overapproximation (ESO) with respect to the sampling $\hat S$ with parameter $v\in \R_{++}^n$:
\begin{assump}[ESO assumption]\label{ass-eso} Let $\hat S$ be a sampling and $v\in \R^n_{++}$ a vector of positive weights. We say that the function $f$ admits an ESO with respect to $\hat{S}$ with parameter $v$, denoted as $(f,\hat S)\sim ESO(v)$, if
the following inequality holds for  all $x,h\in \R^ N$,
\begin{align}\label{a-eso}
 \E[f(x+h_{[\hat S]})]\leq f(x)+\<\nabla f(x),h >_p+\frac{1}{2}\|h\|_{v\circ p}^2,
\end{align}
where $p=(p_1,\dots,p_n)\in\R^n$ is the vector of probabilities associated with $\hat S$ defined in~\eqref{eq:p_i}.
\end{assump} 

Looking behind the compact notation in which the assumption is formulated, observe  that the upper bound is a quadratic function of $h=(h^1,\dots,h^n)$, separable in the blocks $h^i$:
\[f(x)+\<\nabla f(x),h >_p+\frac{1}{2}\|h\|_{v\circ p}^2 \;\; \overset{\eqref{a-wproduct}}{=}\;\; f(x) + \sum_{i=1}^n p_i \left(  \< \nabla_i f(x),  h^i> +  v_i \<\bB_i h^i, h^i >\right).\]

As a tool for the design and analysis of randomized coordinate descent methods, ESO was first formulated in~\cite{PCDM} for the complexity study of  parallel coordinate descent method (PCDM). It  is a powerful technical tool which provides a generic approach to establishing the convergence of  randomized coordinate descent methods of many flavours~\cite{PCDM, Pegasos2, SPCDM, DQA, NSync, Hydra, Hydra2, QUARTZ}. As shown in the listed papers as well as our results which follow,  the convergence of Algorithm~\ref{alg:general} can be established for arbitrary sampling $\hat S$ as long as the parameter vector $v$ is chosen such that $(f,\hat S)\sim ESO(v)$.
Moreover, the vector $v$ appears in the convergence result and directly influences the complexity of the method.

 Since~\cite{PCDM}, the problem of computing efficiently a vector $v$ such that the ESO assumption~\ref{ass-eso} holds has been addressed  in many papers for special uniform samplings relevant to practical implementation including serial sampling, $\tau$-nice sampling~\cite{PCDM,APPROX} and distributed sampling~\cite{Hydra,Hydra2}, and also for a particular example of nonuniform parallel sampling~\cite{NSync}. In this paper we focus on the complexity analysis of Algorithm~\ref{alg:general} and refer the reader to the companion paper~\cite{paper2} for 
 a systematic study of  the computation of admissible vector parameter $v$ for arbitrary sampling $\hat S$.

\section{Simple complexity analysis in the smooth case}\label{sec-smooth}

In this section we give a brief  complexity analysis of ALPHA in the case when $\psi\equiv 0$; that is, when applied to the following unconstrained smooth convex minimization problem:
\begin{equation}\label{eq-prob-smooth}
\begin{array}{ll}
 \mathrm{minimize} &  f(x)\\
\mathrm{subject~to~} & x=(x^1,\dots,x^n)\in \R^{N_1}\times\dots\times \R^{N_n}=\R^N,
\end{array}
\end{equation}

While our general theory, which we develop in Section~\ref{sec-proximalmin}, covers also this special case, the analysis we present here is different and simpler. When applied to problem \ref{eq-prob-smooth}, Step 8 in ALPHA has an explicit solution and the method reduces to Algorithm~\ref{alg:smooth}.

\begin{algorithm}[H]
\begin{algorithmic}[1]
\STATE \textbf{Parameters:}  proper  sampling $\hat S$ with probability vector $p=(p_1,\dots,p_n)$,  vector $v \in \R^n_{++}$
, sequence $\{\theta_k\}_{k\geq 0}$
\STATE \textbf{Initialization:} {choose $x_0\in \R^N$, set $z_0=x_0$}
\FOR{$k\geq 0$} 
\STATE $y_k=(1-\theta_k)x_k+\theta_kz_k$ \\
\STATE Generate a random set of blocks $S_k\sim \hat S $\\
\STATE $z_{k+1}\leftarrow z_k$
\FOR{ $i\in S_k$} 
\STATE
$
z_{k+1}^i = z_k^i - \frac{p_i}{v_i \theta_k} \bB_i^{-1}
\nabla_i f(y_k)$ \\
\ENDFOR\\
\STATE $x_{k+1}=y_k+\theta_k p^{-1}\cdot (z_{k+1}-z_k) $
\ENDFOR
\end{algorithmic}
\caption{ ALPHA specialized to the smooth minimization problem~\eqref{eq-prob-smooth}}
\label{alg:smooth}
\end{algorithm}

We now state the complexity result for ALPHA (Algorithm~\ref{alg:smooth}) in its nonaccelerated variant.

\begin{theo}[ALPHA -- smooth \& nonaccelerated]\label{th-smooth_normal} 
Let $\hat{S}$ be an arbitrary proper sampling and $v\in \R^n_{++}$ be such that $(f,\hat S)\sim ESO(v)$. 
Choose $\theta_k=\theta_0 \in (0,1]$ for all $k\geq 0$. Then for any $y\in \R^N$, the iterates $\{x_k\}_{k\geq 1}$ of Algorithm~\ref{alg:smooth} satisfy: 
\begin{equation}\label{eq-unifijhgsu87s}
\max\left\{ \E[f(\hat{x}_{k})] , \min_{l=1,\dots,k} \E[f(x_{l})]\right\} -f(y) \leq  
\displaystyle\frac{C}{(k-1)\theta_0+1}  \enspace , \forall k\geq 1
\end{equation}
where 
\[\hat{x}_k = \frac{x_k + \theta_0\sum_{l=1}^{k-1} x_l}{1+ (k-1)\theta_0}\]
and 
$$
C = (1-\theta_0)\left(f(x_0)-f(y)\right)+\frac{\theta_0^2}{2}\|x_0-y\|^2_{v\circ p^{-2}}.
$$
In particular, if we choose $\theta_0=\min_i p_i$, then for all $k\geq 1$,
\begin{equation}\label{eq-unifi9878}
\max\left\{\E\left[f\left( \hat{x}_k \right)\right]  , \min_{l=1,\dots, k} \E\left[f(x_{l})\right]\right\} -f(y) \leq  
\frac{(1-\min_i p_i)\left(f(x_0)-f(y)\right)+\frac{1}{2}\|x_0-y\|^2_{ v}}{(k-1)\min_i p_i+1} .\enspace
\end{equation}
\end{theo}

The next result gives a complexity bound for ALPHA (Algorithm~\ref{alg:smooth}) in its accelerated variant.

\begin{theo}[ALPHA -- smooth and accelerated] \label{th-smooth_fast}  
Let $\hat{S}$ be an arbitrary proper sampling and $v\in \R^n_{++}$ be such that $(f,\hat S)\sim ESO(v)$. Choose $\theta_0\in (0,1]$ and define the sequence $\{\theta_k\}_{k\geq 0}$ by \begin{align}\label{a-thetak+1}\theta_{k+1}=\frac{\sqrt{\theta_k^4+4\theta_k^2}-\theta_k^2}{2}.\end{align} Then for any $y\in \R^N$ such that $C\geq 0$, the iterates $\{x_k\}_{k\geq 1}$ of Algorithm~\ref{alg:smooth} satisfy: 
\begin{equation}\label{eq-unifijhgsu8756756s}
\E[f(x_{k})] -f(y)\leq  
\displaystyle\frac{4C}{((k-1)\theta_0+2)^2} , \enspace
\end{equation}
where 
$$
C = (1-\theta_0)\left(f(x_0)-f(y)\right)+\frac{\theta_0^2}{2}\|x_0-y\|^2_{v\circ p^{-2}}.
$$
In particular, if we choose $\theta_0=1$, then for all $k\geq 1$,
\begin{equation}\label{eq-unifidfg9878}
\E[f(x_{k})] -f(y) \leq  
\displaystyle\frac{2\|x_0-y\|^2_{v\circ p^{-2}}}{(k+1)^2}  = \frac{2\sum_{i=1}^n \frac{v_i}{p_i^2}\|x_0^i-y^i\|_i^2}{(k+1)^2}\enspace .
\end{equation}
\end{theo}

In the rest of this section, we provide a short proof of Theorems~\ref{th-smooth_normal} and~\ref{th-smooth_fast}. In Section~\ref{sec-proximalmin} we shall present complexity bounds (Theorem~\ref{th-proximal_normal} and~\ref{th-proximal_fast}) for ALPHA (Algorithm~\ref{alg:general}) as applied to the general regularized problem \eqref{eq-probse}. The proof in the general case is more involved, which is why we prefer to present the smooth case first and also provide a separate briefer proof.

\subsection{Two Lemmas}

We first establish two lemmas and then proceed directly to the proofs of the theorems.

\begin{lemma}\label{lem:expect} For any sampling $\hat{S}$ and any $x,a\in \R^N$ and $w\in \R^n_{++}$, the following identity holds:
\[\|x\|_w^2 - \E \left[\|x+a_{[\hat{S}]}\|_w^2\right]  = \|x\|_{w\circ p}^2 - \|x+a\|_{w\circ p}^2 .\]
\end{lemma}
\begin{proof}
It is sufficient to notice that $
\E \left[\|x+a_{[\hat{S}]}\|_w^2\right] 
=\sum_{i=1}^n \left[ (1-p_i)w_i\|x^i\|_i^2+p_iw_i\|x^i+a^i\|_i^2\right].$
\end{proof}

\begin{lemma} \label{lem:recursion}  Let $\hat{S}$ be an arbitrary proper sampling and  $v\in \R^n_{++}$ be such that $(f,\hat{S})\sim ESO(v)$. Let $\{\theta_k\}_{k\geq 0}$ be an arbitrary  sequence of positive numbers in $(0,1]$ and fix $y\in \R^N$. Then for the sequence of iterates produced by Algorithm~\ref{alg:smooth} and all $k\geq 0$, the  following recursion holds: 
\begin{equation}\label{eq-pr12}
\E_k \left[f(x_{k+1}) + \frac{\theta_k^2}{2}\| z_{k+1}-y\|^2_{ v\circ p^{-2}}\right] \leq 
\left[ f(x_k) +
\frac{\theta_k^2}{2}\|z_k-y\|^2_{ v \circ p^{-2}}\right] - \theta_k (f(x_k) - f(y)) \enspace.
\end{equation}
\end{lemma}
\begin{proof} 
Based on line 8 of Algorithm~\ref{alg:smooth}, we can write
\begin{equation}\label{eq:98hs98hss}a\eqdef \tilde{z}_{k+1}-z_k = - \theta_{k}^{-1}(v^{-1}\circ p) \cdot \bB^{-1}\nabla f(y_k),\end{equation} or equivalently,
$- \nabla f(y_k) = \theta_k (v\circ p^{-1}) \cdot\bB a$. Using this notation, the update on line 10 of Algorithm~\ref{alg:smooth} can be written as
\begin{equation}\label{eq:09us09us}x_{k+1}= y_k + \theta_k p^{-1}\cdot a_{[S_k]}  = y_k + (\theta_k p^{-1}\cdot a)_{[S_k]}.
\end{equation}
Letting $b=\tilde{z}_{k+1}-y$ and $t=\theta_k^2 ( v \circ p^{-1})$, we apply the ESO assumption and rearrange the result:
\begin{eqnarray}\notag
\E_k[f(x_{k+1})]\notag & \overset{\eqref{a-eso}+\eqref{eq:09us09us}}{\leq}& f(y_k) + \<\nabla f(y_k), \theta_k p^{-1}\cdot a>_{p}+ \frac{1}{2}\|\theta_k p^{-1} \cdot a\|_{ v\circ p}^2\\
\notag &\overset{\eqref{a-wproduct} + \eqref{a-norm} +\eqref{eq:98hs98hss}}{=}&  f(y_k) - \frac{1}{2}\|a\|_t^2\\
 &\overset{\eqref{eq:suih9s8hs}}{=} &  f(y_k) - \frac{1}{2}\|b\|_t^2 + \frac{1}{2}\|b-a\|_t^2 + \<\bB a,b-a>_t.\label{eq:jkskd7d5}
\end{eqnarray}
Note that $\|b\|_t^2 = \theta_k^2 \|\tilde{z}_{k+1}-y\|_{v\circ p^{-1}}^2$, $\|b-a\|_t^2 = \theta_k^2\|z_k-y\|_{v\circ p^{-1}}^2$ and
\begin{eqnarray}\notag
\<\bB a,b-a>_t &=& \< -\bB a,a-b >_t \;\;  = \;\; \< \theta_{k}^{-1}(v^{-1}\circ p)\cdot \nabla f(y_k), y-z_k>_t \\
\notag & =& \theta_k \<\nabla f(y_k), y-z_k>\;\;  \overset{\eqref{a-xkykzk}}{=}\;\; \theta_k \ve{\nabla f(y_k)}{ y-y_k}+(1-\theta_k) \ve{\nabla f(y_k)}{ x_k-y_k}\\
\notag &\leq & \theta_k (f(y)-f(y_k)) + (1-\theta_k) (f(x_k)-f(y_k)).
\end{eqnarray}
Substituting  these expressions to \eqref{eq:jkskd7d5}, we obtain the recursion:
\begin{equation}\label{eq:siusuhs}
\E_k[f(x_{k+1})] \leq 
\theta_k f(y)+(1-\theta_k) f(x_k) +
\frac{\theta_k^2}{2}\|z_k-y\|^2_{ v \circ p^{-1}}-\frac{\theta_k^2}{2}\|\tilde z_{k+1}-y\|^2_{ v\circ p^{-1}}.
\end{equation}
It now only remains to apply Lemma~\ref{lem:expect} to the last two terms in \eqref{eq:siusuhs},  with $x\leftarrow z_k-y$, $w\leftarrow v\circ p^{-2}$ and $\hat{S}\leftarrow S_k$, and rearrange the resulting inequality.
\end{proof}

\subsection{Proof of Theorem~\ref{th-smooth_normal} }\label{sec-th-smooth-normal}

Using the fact that $\theta_k=\theta_0$, for all $k$ and taking expectation in both sides of \eqref{eq-pr12}, we obtain the recursion
$$
 \phi_{k+1} + \theta_0^2 r_{k+1} 
\leq (1-\theta_0)\phi_k + \theta_0^2 r_k,\quad k\geq 0, 
$$
where $\phi_{k}\eqdef \E[f(x_k)]-f(y)$ and $r_k \eqdef \tfrac{1}{2}\E[\|z_{k}-y\|^2_{v\circ p^{-2}}]$. Combining these inequalities, we get
\begin{equation}\label{eq:0s9j098d}(1+\theta_0(k-1)) \min_{l=1,\dots,k} \phi_l \leq \phi_k + \theta_0 \sum_{l=1}^{k-1} \phi_l \leq (1-\theta_0)\phi_0 + \theta_0^2 r_0.\end{equation}
Let $\alpha_k = 1+(k-1)\theta_0$. By convexity,
\[f(\hat{x}_k) = f\left(\frac{x_k + \sum_{l=1}^{k-1}\theta_0 x_l}{\alpha_k}\right) \leq \frac{f(x_k) + \sum_{l=1}^{k-1}\theta_0 f(x_l)}{\alpha_k}.\]
Finally, subtracting $f(y)$ from both sides  and taking expectations,  we obtain
\[\E[f(\hat{x}_k)]-f(y) \leq \frac{\phi_k + \sum_{l=1}^{k-1}\theta_0 \phi_l}{\alpha_k} \overset{\eqref{eq:0s9j098d}}{\leq} \frac{(1-\theta_0)\phi_0 +\theta_0^2 r_0}{\alpha_k}.\]

\subsection{Proof of Theorem~\ref{th-smooth_fast} }\label{sec-th-smooth-fast}

If $\theta_0\in (0,1]$, then the  sequence $\{\theta_k\}_{k\geq 0}$ has the following properties (see~\cite{tseng2008accelerated}):
\begin{align}\label{a-thetak20}
0< \theta_{k+1}\leq \theta_k\leq \frac{2}{k+2/\theta_0} \leq 1,\\\label{a-thetak30}
\frac{1-\theta_{k+1}}{\theta_{k+1}^2}=\frac{1}{\theta_k^2}.
\end{align}
After dividing both sides of~\eqref{eq-pr12} by $\theta_k^2$,  using~\eqref{a-thetak30} and taking expectations, we obtain:
\begin{equation}\label{a-eklf}
\frac{1-\theta_{k+1}}{\theta_{k+1}^2} \phi_{k+1} +r_{k+1}
\leq \frac{1-\theta_k}{\theta_k^2} \phi_k + r_k \leq \frac{1-\theta_0}{\theta_0^2} \phi_0 + r_0,
\end{equation}
where $\phi_{k}$ and $r_k$ are as in the  proof of Theorem~\ref{th-smooth_normal}.  Finally,
\begin{eqnarray*}\phi_k &\overset{\eqref{a-thetak30}}{=} & \frac{(1-\theta_k)\theta_{k-1}^2}{\theta_k^2} \phi_k 
\;\;\leq\;\;  \frac{(1-\theta_k)\theta_{k-1}^2}{\theta_k^2} \phi_k + \theta_{k-1}^2 r_k \;\; \overset{\eqref{a-eklf}}{\leq} \;\; \frac{(1-\theta_0)\theta_{k-1}^2}{\theta_0^2} \phi_0 + \theta_{k-1}^2 r_0 \\
&= & \frac{\theta_{k-1}^2}{\theta_0^2}\left((1-\theta_0)\phi_0 + \theta_{0}^2 r_0\right)\;\; = \;\; \frac{\theta_{k-1}^2}{\theta_0^2} C\;\;\overset{\eqref{a-thetak20}}{\leq} \;\;
\frac{4C}{((k-1)\theta_0+2)^2}.\end{eqnarray*}
Note that in the last inequality we used the assumption that $C\geq 0$.

\section{The many variants of ALPHA (in the smooth case)}\label{sec-specialcasesm}

The purpose of this section is to demonstrate that ALPHA is a very flexible method, encoding several classical as well as modern optimization methods for special choices of the parameters of the method. In order to achieve this goal, it is enough to focus on the smooth case, i.e., on Algorithm~\ref{alg:smooth}. Similar reasoning can be applied to the  proximal case.

Note that in ALPHA we have the liberty to {\em choose} the sampling $\hat{S}$ and the sequence $\{\theta_k\}_{k\geq 0}$. 
 As we have already seen, by modifying the sequence we can obtain {\em simple} (i.e., nonaccelerated) and {\em accelerated} variants of the method. By the choice of the sampling, we can  force the method to be {\em deterministic} or {\em randomized}. In the latter case, there are many ways of choosing the distribution of the sampling. Here we will constrain ourselves to a basic classification between {\em uniform samplings} (samplings for which $p_i=p_{i'}$ for all $i\in [n]$) and non-uniform or {\em importance samplings}.
This is summarized in Table~\ref{tab-1}.

\begin{table}[!ht]
\footnotesize
\centering
\begin{tabular}{| >{\centering\arraybackslash}m{0.13\textwidth} | >{\centering\arraybackslash}m{0.14\textwidth}|>{\centering\arraybackslash}m{0.10\textwidth}|>{\centering\arraybackslash}m{0.12\textwidth}| 
>{\centering\arraybackslash}m{0.09\textwidth} | >{\centering\arraybackslash}m{0.17\textwidth} | }
\hline 
{Parameters of Algorithm~\ref{alg:smooth}}
 &	\multicolumn{3}{c}{$\hat S$} 
& \multicolumn{2}{|c|}{$\theta_k$} 
\\
\cline{1-6} \multirow{2}{*}{Setting}
&  \multirow{2}{*}{$|\hat S|=n$} & 
\multicolumn{2}{c|}{{$|\hat S|<n$ }} & \multirow{2}{*}{ \small$\theta_{k+1}=\theta_k$} &
\multirow{2}{*}{\small$\theta_{k+1}=\frac{\sqrt{\theta_k^4+4\theta_k^2}-\theta_k^2}{2}$}  \\ \cline{3-4}
 &  &
    $p_i=p_{i'}$, \small$ \forall i,i'\in[n]$ & $p_i \neq p_{i'}$, \small$ \exists i,i'\in[n]$  & & \\
\cline{1-6} \multirow{2}{*}{{Characteristic}} &
  \multirow{2}{*}{\small Deterministic} & 
\multicolumn{2}{c|}{\small{Randomized }} & \multirow{2}{*}{ \small Simple} &
\multirow{2}{*}{\small Accelerated}  \\ \cline{3-4}  &  &
    \small Uniform sampling & \small Importance sampling   & & \\ 
\hline 
\multicolumn{1}{c}{Special cases} &
\multicolumn{5}{c}{~~} \\ 
\hline  {Algorithm~\ref{alg:gd}} 
& \cmark &  & &  \cmark & \\
\cline{1-6} Algorithm~\ref{alg:accgd} 
&  \cmark &  & &   & \cmark \\
\cline{1-6}  
{Algorithm~\ref{alg:rcdm}} 
 & \cmark  & \cmark &  \cmark  & \cmark & \\
\cline{1-6} 
{Algorithm~\ref{alg:accsmooth}} 
& \cmark  &  \cmark & \cmark &   & \cmark \\
\hline 
\end{tabular}
\caption{Special cases of Algorithm~\ref{alg:smooth}.}\label{tab-1}
\end{table}

The deterministic variants of ALPHA (Algorithm~\ref{alg:gd} and~\ref{alg:accgd}) are obtained by choosing the sampling which always selects all blocks: $\hat S=[n]$ with probability 1.
The ESO assumption in this special  case has the form
\begin{align}\label{a-esofull1}
f(x+h)\leq f(x)+\ve{\nabla f(x)}{h}+\frac{1}{2}\|h\|_v^2,\enspace \forall x,h\in \R^N,
\end{align}
which simply requires the  gradient of $f$ to be $1$-Lipschitz with respect to the norm $\|\cdot\|_v$. Note that $\|h\|_v^2 = h^\top \tilde{\bB} h$, where $\tilde{\bB}$ is the block diagonal matrix defined by $\tilde \bB\eqdef \Diag(v_1\bB_1,\dots,v_n\bB_n)$. Likewise, if there is just a single block in our block setup (i.e., if $n=1$) then it is natural to only consider a sampling which picks this block with probability 1, which again results in a deterministic method. However, in this case the norm $\|\cdot\|_v$ can be  an arbitrary Euclidean norm (that is, it does not need to be block diagonal).

In the randomized variants of ALPHA (Algorithm~\ref{alg:rcdm} and~\ref{alg:accsmooth}) we allow for the sampling to have an  {\em arbitrary distribution}. 


\subsection{Special case 1: gradient descent}

By specializing Algorithm~\ref{alg:smooth} to the choice $\hat S=[n]$ and $\theta_k=1$ for all $k$, we obtain classical gradient descent (with fixed stepsize).  Indeed, note that in this special case we have
\begin{align}\label{a-xyz}
x_k=y_k=z_k,\qquad \forall k\geq 1.
\end{align}
Recall that the ESO assumption reduces to~\eqref{a-esofull1} when $\hat S=[n]$.
\begin{algorithm}[H]
\begin{algorithmic}[1]
\STATE \textbf{Parameters:} vector $v\in \R_{++}^ n$ such that~\eqref{a-esofull1} holds
\STATE \textbf{Initialization:} choose $x_0\in \R^N$
\FOR{$k\geq 0$} 
\FOR{ $i\in [n]$} 
\STATE
$
x_{k+1}^i=x_k^i- \frac{1}{v_i} \bB_i^{-1}
\nabla_i f(x_k)$ \\
\ENDFOR\\
\ENDFOR
\end{algorithmic}
\caption{Gradient Descent (GD) for solving~\eqref{eq-prob-smooth}}
\label{alg:gd}
\end{algorithm}
The complexity of the method is a corollary of Theorem~\ref{th-smooth_normal}.

\begin{coro}\label{coro-gd}
 For any optimal solution $x_*$ of~\eqref{eq-prob-smooth}, the output of Algorithm~\ref{alg:gd} for 
all $k\geq 1$ satisfies: 
\begin{align}\label{a-fxkfstarleq}
f(x_{k})-f(x_*)\leq  
\displaystyle\frac{\|x_0-x_*\|^2_{v}}{2k}.
\end{align}
In particular, for $\epsilon>0 $, if 
\begin{align}\label{a-coro-gd}
k\geq \frac{\|x_0-x_*\|^2_{ v}}{2\epsilon},
\end{align}
then $ f(x_k)-f(x_*)\leq \epsilon$.
\end{coro}
\begin{proof}
By letting $y=x_k$ in~\eqref{eq-pr12}  we know that:
$$
f(x_{k+1}) \leq f(x_k)+\frac{\theta_k^2}{2} \|z_k-x_k\|^2 \overset{\eqref{a-xyz}}{=} f(x_k),\enspace \forall k\geq 1.
$$
Note that for this special case $x_k=z_k$. Therefore,
$$
f(x_k)-f(x_*)=\min_{l=1,\dots,k} f(x_{l})-f(x_*)\leq  
\displaystyle\frac{\|x_0-x_*\|^2_{v}}{2k},\enspace \forall k\geq 1,
$$
where the second inequality follows from applying Theorem~\ref{th-smooth_normal} to 
$\theta_0=1$ and $\hat S=[n]$.
\end{proof}
 Corollary~\ref{coro-gd} is a basic result and can be found in many  textbooks on convex optimization; see for example~\cite{NesterovIntro}.  

\subsection{Special case 2: accelerated gradient descent}\label{subsec-agd}

Let us still keep $\hat S=[n]$, but assume now the sequence  $\{\theta_k\}_{k\geq 0}$ is chosen according to~\eqref{a-thetak+1}. In this case,  Algorithm~\ref{alg:smooth} reduces to  accelerated gradient descent.

\begin{algorithm}[H]
\begin{algorithmic}[1]
\STATE \textbf{Parameters:} positive vector $v\in \R_{++}^n$  such that~\eqref{a-esofull1} holds
\STATE \textbf{Initialization:} {choose $x_0\in \R^N$, set $z_0=x_0$ and $\theta_0=1$}
\FOR{$k\geq 0$} 
\FOR{ $i\in [n]$} 
\STATE
$
z_{k+1}^i=z_k^i - \frac{1}{v_i \theta_k} \bB_i^{-1}
\nabla_i f((1-\theta_k)x_k+\theta_kz_k)
$ \\
\ENDFOR\\
\STATE $x_{k+1}=(1-\theta_k)x_k+\theta_k z_{k+1}$\\
\STATE $\theta_{k+1}=\frac{\sqrt{\theta_k^4+4\theta_k^2}-\theta_k^2}{2}$
\ENDFOR
\end{algorithmic}
\caption{Accelerated Gradient Descent (AGD) for solving~\eqref{eq-prob-smooth}}
\label{alg:accgd}
\end{algorithm}
Note that only two sequences $\{x_k,z_k\}_{k\geq 0}$ are explicitly used in Algorithm~\ref{alg:accgd}. This is achieved by replacing $y_k$  in Algorithm~\ref{alg:smooth} by $(1-\theta_k) x_k+\theta_k z_k$.   The following result  follows directly from Theorem~\ref{th-smooth_fast} by letting
$\theta_0=1$ and $p_i=1$ for all $i\in [n]$.
\begin{coro}\label{coro-accgd}
For any optimal solution $x_*$ of~\eqref{eq-prob-smooth}, the output of Algorithm~\ref{alg:accgd} for 
all $k\geq 1$ satisfies: 
\begin{align}\label{a-dfd}
f(x_k)-f(x_*)\leq \frac{2\|x_0-x_*\|^2_{v}}{(k+1)^2}.
\end{align}
In particular, for $\epsilon>0 $, if 
\begin{align}\label{a-coro-accgd}
k\geq \sqrt{\frac{2\|x_0-x_*\|^2_{ v}}{\epsilon}}-1,
\end{align}
then $f(x_k)-f(x_*)\leq \epsilon$.
\end{coro}
Algorithm~\ref{alg:accgd} is a special case of Algorithm 1 in~\cite{tseng2008accelerated}. The complexity bound~\eqref{a-dfd}  was also proved in~\cite[Corollary 1]{tseng2008accelerated}. See also   \cite{APPROX}.

\subsection{Special case 3: Parallel coordinate descent }
We now allow the method (Algorithm~\ref{alg:smooth}) to use an arbitrary sampling $\hat S$, but keep $\theta_k=\theta_0$ for all $k\geq 1$. This leads to Algorithm~\ref{alg:rcdm}.

\begin{algorithm}[H]
\begin{algorithmic}[1]
\STATE \textbf{Parameters:}   proper sampling $\hat S$ with probability vector $p=(p_1,\dots,p_n)$,  vector $v\in \R_{++}^ n$ for which $(f,\hat S)\sim ESO(v)$ 
\STATE \textbf{Initialization:} {choose $x_0\in \R^N$, set $z_0=x_0$ and $\theta_0=\min_i p_i$}
\FOR{$k\geq 0$} 
\STATE $y_k=(1-\theta_0)x_k+\theta_0z_k$ \\
\STATE Generate a random set of blocks $S_k\sim \hat S $\\
\STATE $z_{k+1}\leftarrow z_k$
\FOR{ $i\in S_k$} 
\STATE
$
z_{k+1}^i = z_k^i - \frac{p_i}{v_i \theta_0 } \bB_i^{-1}
\nabla_i f(y_k)$ \\
\ENDFOR\\
\STATE $x_{k+1}=y_k+{\theta_0}p^{-1} \cdot(z_{k+1}-z_k) $
\ENDFOR
\end{algorithmic}
\caption{Parallel Coordinate Descent (PCD) for solving~\eqref{eq-prob-smooth}}
\label{alg:rcdm}
\end{algorithm}
Note that in classical non-accelerated coordinate descent methods, only a single  sequence of iterates is needed. This is indeed the case for our method as well, in the  special case when the sampling $\hat S$ is uniform and $\theta_0=\E[|\hat S|]/n$, so that the three sequences are equal to each other.   We now state a direct corollary of Theorem~\ref{th-smooth_normal}.

\begin{coro}\label{coro-rcdm} 
For any optimal solution $x_*$ of~\eqref{eq-prob-smooth}, the output of Algorithm~\ref{alg:rcdm} for 
all $k\geq 1$ satisfies: 
\begin{equation*}
\begin{split}
&\max\left\{\E\left[f\left( \tfrac{x_k + \min_i p_i\sum_{l=1}^{k-1} x_l}{1+ (k-1)\min_i p_i} \right)\right]  , \min_{l=1,\dots, k} \E\left[f(x_{l})\right]\right\} -f(x_*) \\ &\leq  
\displaystyle\frac{(1-\min_i p_i)\left(f(x_0)-f(x_*)\right)+\frac{1}{2}\|x_0-x_*\|^2_{ v}}{(k-1)\min_i p_i+1} .\enspace
\end{split}
\end{equation*}
In particular, for $\epsilon>0 $, if 
\begin{align}\label{a-coro-rcdm}
k\geq \frac{(1-\min_i p_i)\left(f(x_0)-f(x_*)\right)+\frac{1}{2}\|x_0-x_*\|^2_{ v}}{\min_i p_i \epsilon}-\frac{1}{\min_i p_i}+1,
\end{align}
then
$$
\max\left\{\E\left[f\left( \tfrac{x_k + \min_i p_i\sum_{l=1}^{k-1} x_l}{1+ (k-1)\min_i p_i} \right)\right]  , \min_{l=1,\dots, k} \E\left[f(x_{l})\right]\right\} -f(x_*)\leq \epsilon.
$$
\end{coro}

In the special case when $\hat{S}$ is the serial uniform sampling,  the three sequences $\{x_k,y_k,z_k\}_{k\geq 0}$ coincide, and one can show that the following bound holds:
$$
\E[f(x_{k})]-f(x_*)\leq  
\displaystyle\frac{n}{k-1+n}\left[
\left(1-\frac{1}{n}\right)\left(f(x_0)-f(x_*)\right)+\frac{1}{2}\|x_0-x_*\|^2_{v}\right].
$$
Randomized coordinate descent with serial and importance sampling (in a form different from Algorithm~\ref{alg:rcdm}) was considered by Nesterov~\cite{Nesterov:2010RCDM}. In the special case of serial uniform sampling when Algorithm~\ref{alg:rcdm} is the same as in~\cite{Nesterov:2010RCDM}, the following convergence rate was proved in~\cite{Nesterov:2010RCDM}:
$$
\E[f(x_k)]-f(x_*) \leq\frac{2n}{k+4}
 \cR^2 (x_0)
$$
where $\cR(x_0)$ is a weighted level-set distance to the set of optimal points $X_*$:
$$
\cR(x_0)\eqdef \max_x \left\{ \max_{x_* \in X_*} \|x_0-x_*\|_v^2 \;:\; f(x) \leq f(x_0)\right\}.
$$

Our result does not require the level sets of $f$ to be bounded.

\subsection{Special case 4: Accelerated parallel coordinate descent }

To obtain the accelerated coordinate descent method, as a special case of Algorithm~\ref{alg:smooth}, we only need to let the sequence $\{\theta_k\}_{k\geq 0}$ satisfy~\eqref{a-thetak+1}.

\begin{algorithm}[H]
\begin{algorithmic}[1]
\STATE \textbf{Parameters:}  proper  sampling $\hat S$ with probability vector $p=(p_1,\dots,p_n)$,  vector $v \in \R^n_{++}$ for which $(f,\hat{S})\sim ESO(v)$, $\theta_0=1$
\STATE \textbf{Initialization:} {choose $x_0\in  \R^N$ and set  $z_0=x_0$}
\FOR{$k\geq 0$} 
\STATE $y_k=(1-\theta_k)x_k+\theta_kz_k$ \\
\STATE Generate a random set of blocks $S_k\sim \hat S $\\
\STATE $z_{k+1}\leftarrow z_k$
\FOR{ $i\in S_k$} 
\STATE
$
z_{k+1}^i = z_k^i - \frac{p_i}{v_i \theta_k} \bB_i^{-1}
\nabla_i f(y_k)$ \\
\ENDFOR\\
\STATE $x_{k+1}=y_k+\theta_k p^{-1} \cdot (z_{k+1}-z_k) $\\
\STATE $\theta_{k+1}=\frac{\sqrt{\theta_k^4+4\theta_k^2}-\theta_k^2}{2}$
\ENDFOR
\end{algorithmic}
\caption{Accelerated parallel coordinate descent (APCD) for solving~\eqref{eq-prob-smooth}}
\label{alg:accsmooth}
\end{algorithm}
The convergence result then follows directly as a corollary of Theorem~\ref{th-smooth_fast}.
\begin{coro}\label{coro-accsmooth} 
For any optimal solution $x_*$ of~\eqref{eq-prob-smooth}, the output of Algorithm~\ref{alg:accsmooth} for 
all $k\geq 1$ satisfies: 
\begin{equation}\label{eq-coro-as}
 \E\left[f(x_{k})\right] -f(x_*)\leq  
\displaystyle\frac{2\|x_0-x_*\|^2_{v\circ p^{-2}}}{(k+1)^2}.
\end{equation}
In particular, for $\epsilon>0 $, if 
\begin{align}\label{a-coro-accs}
k\geq \sqrt{\frac{2\|x_0-x_*\|^2_{v\circ p^{-2}}}{\epsilon}}-1,
\end{align}
then $\E\left[f(x_k)\right]-f(x_*) \leq \epsilon$.
\end{coro}
When specialized to the serial uniform sampling (sampling $\hat{S}$ for which $\Prob(\hat{S}=\{i\})=1/n$ for $i\in [n]$), the bound~\ref{eq-coro-as} simplifies to:
\begin{equation}\label{eq-coro-as22}
 \E\left[f(x_{k})\right] -f(x_*)\leq  
\displaystyle\frac{2n^2\|x_0-x_*\|^2_{v}}{(k+1)^2}.
\end{equation}
An accelerated coordinate descent method for unconstrained minimization in the special case of serial uniform sampling was first proposed and analyzed by Nesterov~\cite{Nesterov:2010RCDM}, where the following bound was proved:
 \begin{equation}\label{eq-coro-as-nesterov}
 \E\left[f(x_{k})\right] -f(x_*)\leq  
\displaystyle \left(\frac{n}{k+1}\right)^2\left[2\|x_0-x_*\|^2_{v}+\frac{1}{n^2}\left(f(x_0)-f(x_*)\right)\right].
\end{equation}
Comparing~\eqref{eq-coro-as22} and~\eqref{eq-coro-as-nesterov}, it is clear that we obtain a  better  bound. An accelerated coordinate descent method (APPROX) utilizing an {\em arbitrary uniform sampling} was studied by Fercoq and Richt\'{a}rik \cite{APPROX}. Algorithm~\ref{alg:accsmooth}, when restricted to a uniform sampling, is similar to this method. The main difference is  in the value of $\theta_0$. Indeed, in~\cite{APPROX}, $\theta_0$ is chosen to be $\E[|\hat S|]/n$, while our analysis allows $\theta_0$ to be chosen as large as $1$. This lead to larger stepsizes and a simplified and improved convergence bound.

Each serial sampling $\hat S$ is uniquely characterized by the vector of probabilities $p=(p_1,\dots,p_n)$ where $p_i$ is defined by~\eqref{eq:p_i}. Suppose that the function $f$ has block-Lipschitz gradient with constants $L_1,\dots,L_n$:
\begin{align}\label{a-fxbUi}
f(x+\bU_i h^i)\leq f(x)+\<\nabla_i f(x), h^i>+
\frac{L_i}{2} \|h^i\|_i^2,\qquad  \forall i\in [n],\; h^i\in \R^{N_i}, \;x\in \R^N.
\end{align}
If $\hat S$ is a serial sampling, then
\begin{align*}
\E\left[f(x+ h_{[\hat S]})\right]&=\sum_{i=1}^n p_i f(x+\bU_i h^i) \overset{\eqref{a-fxbUi}}{\leq} 
f(x)+\sum_{i=1}^n p_i \ve{\nabla_i f(x)}{h^i}+\sum_{i=1}^n \frac{p_iL_i}{2} \|h^i\|_i^2 
\\ &= f(x)+\ve{\nabla f(x)}{h}_p +\frac{1}{2}\|h\|_{L\circ p}^2,
\end{align*}
which means that $(f,\hat S)\sim ESO (L)$, where $L=(L_1,\dots,L_n)\in \R_{++}^n$. We can now find a sampling $\hat{S}$ for which the complexity bound \eqref{a-coro-accs} is minimized. This leads to the choice:
\begin{align}\label{a-optp}
p_i^*\eqdef \frac{(L_i\|x_*^i-x_0^i\|_i^2)^{\frac{1}{3}}}{\displaystyle\sum_{j=1}^n (L_j\|x_*^j-x_0^j\|_j^2)^{\frac{1}{3}}},\qquad i=1,\dots,n.
\end{align}

The optimal serial sampling given by ~\eqref{a-optp} is not very useful without (at least some) knowledge of $x_*$, which is not known. However, note that the formula \eqref{a-optp} confirms the intuition that blocks with larger $L_i$ and larger distance to the optimal block $\|x_*^i-x_0^i\|_i$ should be picked (and hence updated) more often.


\section{Efficient implementation}\label{sec-eie}
As mentioned in Section~\ref{sec-desalgo}, Algorithm~\ref{alg:general} requires full-dimensional operations at each iteration unless $\theta_k p_i=1$ for all $i\in [n]$ and $k\in \bN$. In this section we provide an equivalent form of Algorithm~\ref{alg:general} which is suitable for efficient implementation under some additional assumptions on the computation of the gradient $\nabla f$.

\subsection{Equivalent form}
Focusing on the iterates $x_k,y_k,z_k$ in Algorithm~\ref{alg:general} only, the general algorithm can schematically be written as follows:
\begin{eqnarray}
y_k &\leftarrow & (1-\theta_k)x_k + \theta_k z_k\label{eq:97987}\\
z_{k+1} &\leftarrow & \left\{\begin{array}{ll} \argmin_{z\in \R^{N_i}} \big\{\< \nabla_i f(y_k), z>+\frac{\theta_k  v_i }{2p_i} \|z-z_k^i\|_i^2+\psi^i(z) \big\} & i\in S_k \\ z_k^i & i\notin S_k\end{array}\right. \label{eq:9ns8ijd}\\
x_{k+1} &\leftarrow & y_k + \theta_k p^{-1}\cdot(z_{k+1}-z_k)  \label{eq:d9hdujd}
\end{eqnarray}

Consider the {change of variables} from $\{x_k,y_k,z_k\}$ to $\{z_k,g_k\}$   where
\begin{equation}\label{eq:suihidd} 
g_k = \alpha_k^{-1}(y_k-z_k)
\end{equation}
and $\{\alpha^k\}_{k\geq 0}$ is a sequence defined by:
\begin{align}\label{a-alphak}
\alpha_0=1,\enspace
\alpha_{k}=(1-\theta_k)\alpha_{k-1},\enspace \forall k\geq 1.
\end{align}
Note that, in all the special cases presented in Section~\ref{sec-specialcasesm} and~\ref{sec-proximalmin}, either $
\theta_k < 1
$ for all $k \geq 1$ or $\theta_k=1$ for 
all $k\geq 1$. The latter case does not require the full-dimensional operation $y_k=(1-\theta_k)x_k+\theta_k z_k$ because 
the three sequences equal to each other.
We thus only address the case when $
\theta_k < 1
$ for all $k \geq 1$ which implies that $\alpha_k \neq 0 $ for all $k\geq 1$.

Then from $\{z_k,g_k\}$ and $\{\alpha_k\}$ we can recover $\{x_k,y_k\}$ as follows:
\begin{equation}\label{eq:s98ddxmk}
y_k \overset{\eqref{eq:suihidd}}{=} z_k + \alpha_k g_k, \qquad
x_{k+1} \overset{\eqref{eq:d9hdujd}+\eqref{eq:suihidd}}{=} (z_k+ \alpha_k g_k) + \theta_k p^{-1} \cdot (z_{k+1}-z_k) . \end{equation}
Moreover,  $g_{k+1}$ can be computed recursively  as follows:
\begin{eqnarray*}g_{k+1} &\overset{\eqref{eq:suihidd}}{=}& \alpha_{k+1}^{-1}( y_{k+1}-z_{k+1}) \overset{\eqref{eq:97987}}{=} \alpha_{k+1}^{-1}(1-\theta_{k+1})(x_{k+1}-z_{k+1})\\&
 \overset{\eqref{a-alphak}}{=} &\alpha_{k}^{-1}(x_{k+1}-z_{k+1})
 \overset{\eqref{eq:s98ddxmk}}{=}  g_k - \alpha_k^{-1}(e-\theta_k p^{-1})\cdot (z_{k+1}-z_k),\end{eqnarray*}
where $e\in \R^n$ is the vector of all ones.
Therefore the updating scheme \eqref{eq:97987}--\eqref{eq:d9hdujd} can thus be written in the form:
\begin{eqnarray}
z_{k+1} &\leftarrow & \left\{\begin{array}{ll} \argmin_{z\in \R^{N_i}} \big\{\< \nabla_i f(\alpha_k g_k+z_k), z>+\frac{\theta_k  v_i }{2p_i} \|z-z_k^i\|_i^2+\psi^i(z) \big\} & i\in S_k \\ z_k^i & i\notin S_k\end{array}\right.\label{eq:sj8duhdus}\\
g_{k+1}&\leftarrow & g_k -  \alpha_k^{-1}(e-\theta_k  p^{-1})\cdot(z_{k+1}-z_k)
\\ \alpha_{k+1} &\leftarrow & (1-\theta_{k+1})\alpha_k
\label{eq:098456hs9}
\end{eqnarray}
Hence Algorithm~\ref{alg:general} can be written in the following equivalent form.
\begin{algorithm}[H]
\begin{algorithmic}[1]
\STATE \textbf{Parameters:}  proper  sampling $\hat S$ with probability vector $p=(p_1,\dots,p_n)$,   $v \in \R^n_{++}$, sequence $\{\theta_k\}_{k\geq 0}$
\STATE \textbf{Initialization:} {choose $x_0\in \dom \psi$, set  $z_0=x_0$, $g_0=0$ and $\alpha_0=1$}
\FOR{ $k \geq 0$}
\STATE Generate a random set of blocks $S_k \sim \hat{S}$
\STATE  $z_{k+1} \leftarrow z_k$ , $g_{k+1}\leftarrow g_k$
\FOR{ $i \in S_k$}
\STATE  $t_{k}^{i} = \arg\min_{t \in \R^{N_i}} \left\{ \langle \nabla_i f (\alpha_k g_k + z_k), t \rangle + \frac{\theta_k   v_i}{2 p_i}  \|t\|_{i}^2 + \psi^i(z_{k}^{i}+t)\right\}$ 
\STATE $z_{k+1}^{i} \leftarrow z_{k}^{i} + t_{k}^{i}$
\STATE $g_{k+1}^{i} \leftarrow g_{k}^{i} - \alpha_k^{-1}(1-\theta_k p_i^{-1})t_{k}^{i}$
\STATE $\alpha_{k+1}=(1-\theta_{k+1})\alpha_k$
\ENDFOR 
\ENDFOR 
\STATE OUTPUT: $x_{k+1} =  z_k+ \alpha_k g_k + \theta_k  p^{-1}\cdot(z_{k+1}-z_k)$
\end{algorithmic}
\caption{Efficient equivalent of Algorithm~\ref{alg:general}}\label{alg:efficientequivalent}
\end{algorithm}

\subsection{Cost of a single iteration}

In order to perform Step 7,
  it is important that we  have access to $\nabla_i f(y_k)=\nabla_i f(\alpha_k g_k+z_k)$ 
  \textit{without} actually computing $y_k$. In~\cite{APPROX}, the authors show that this is possible for problems~\eqref{eq-probse} where
   $f$ can be written as:
  \begin{align}\label{eq-efff}
f(x)=\sum_{j=1}^m  \phi_j(e_j^\top \bA x),\enspace \forall x\in \R^N, 
  \end{align}
  for some matrix $\bA \in \R^{m\times N}$. 
Let us write 
$$
u_k= \bA g_k,\enspace w_k=\bA z_k,\enspace k\geq 1.
$$  
For $i\in[n]$ and $j\in [m]$ denote by $\bA_{ji}$
the $i$th block of the $j$th row vector of the matrix $\bA$, i.e.,
$\bA_{ji}=U_i^\top \bA^\top e_j$. For each $i\in [n]$, denote by $I_i$ the number of rows
 containing a non-zero $i$th block, i.e.,
$$
I_i\eqdef \{j\in [m]: \bA_{ji}\neq 0\}.
$$
 Then for $f$ taking the form of~\eqref{eq-efff} we have
  $$
\nabla_i f(y_k)=\nabla_i f(\alpha_k g_k+z_k)
=\sum_{j=1}^m \bA_{ji} \phi_j'(\alpha_k u_k^j+  w_k^j)  =\sum_{j\in I_i }\bA_{ji} \phi_j'(\alpha_k u_k^j+  w_k^j),
  $$
  where by abuse of notation $u_k^j$ and $w_k^j$ denote respectively the $j$th element of the vectors $u_k$ and $w_k$. With the knowledge of the vectors $u_k$ and $w_k$, computing $\nabla_i f(y_k)$ requires 
  $O(|I_iN_i|)$ operations. Now in order to keep record of the vectors $u_k$ and $w_k$, we use the following equality:
  \begin{align}\label{a-wk+1}
w_{k+1}=\bA z_{k+1}=\bA z_k + \bA (z_{k+1}-z_k) =w_k+ \sum_{i \in S_k}
A\bU_i t_k^i ,
  \end{align}
  and
  \begin{align}\label{a-uk+1}
u_{k+1}=\bA g_{k+1}=\bA g_k +\bA(g_{k+1}-g_k)
=  u_k+\sum_{i \in S_k}\alpha_k^{-1}(1-\theta_kp_i^{-1})
\bA\bU_i t_k^i .
  \end{align}
  Since $\bA\bU_i$ is a matrix with $|I_i N_i|$ nonzero elements,  the updating  schemes~\eqref{a-wk+1} and~\eqref{a-uk+1} require then
  $$
\sum_{i\in S_k} O(I_iN_i)   
  $$
  operations. Note that this is also the total complexity of gradient computation $\nabla_i f(y_k)$ at $k$th iteration. Denote by $\nnz(\bA)$ the total number of nonzero blocks of the matrix $A$, i.e.,
$$\nnz(\bA)\eqdef \sum_{i=1}^n I_i.$$
Let us consider the special case when $|\hat S|=\tau$ and $N_i=N/n$ for all $i\in [n]$. In this case, the expected one iteration computational complexity is:
$$
\E\left[\sum_{i\in S_k} O\left(I_i \frac{N}{n}\right)\right]=O\left( \sum_{i=1}^n \frac{\tau N}{n^2} I_i\right)=O\left(\frac{\tau N\nnz(\bA)}{n^2}\right) .
$$
To make it more direct to understand, let us consider the case when each block contains only one coordinate, i.e., $N=n$. Then the latter expected one iteration complexity becomes
$$
\E\left[\sum_{i\in S_k} O\left(I_i \frac{N}{n}\right)\right]=O\left(\frac{\tau \nnz(\bA)}{n}\right).
$$
Hence, in this case the one iteration complexity in expectation of Algorithm~\ref{alg:efficientequivalent} is of order $O(\tau \bar \omega)$ where $\bar \omega$ is the average number of nonzero elements of the columns of $\bA$. Not considering the time spent on synchronization and handling read/write conflicts, the average processing time would be $O(\bar \omega)$  if we use a parallel implementation with $\tau$ processors.


\section{Proximal  minimization}\label{sec-proximalmin}

In this section  we present and prove complexity results for ALPHA (Algorithm~\ref{alg:general}) as applied to the general problem  \eqref{eq-probse} involving the proximal term. We leave the discussion concerning special cases to the appendix.

\subsection{Complexity results}

In the presence of the proximal term $\psi$, the same complexity bounds as those given in Theorems~\ref{th-smooth_normal} and~\ref{th-smooth_fast} hold for the output of Algorithm~\ref{alg:general}, with the exception that $\theta_0$ is only allowed to be chosen between $(0,\min_i p_i]$. We now state the formal complexity theorems, first in the nonaccelerated and then in the accelerated case.

\begin{theo}[ALPHA -- proximal \& nonaccelerated]\label{th-proximal_normal} 
Let $\hat{S}$ be arbitrary proper sampling and $v\in \R^n_{++}$ be such that $(f,\hat S)\sim ESO(v)$. 
Choose $\theta_k=\theta_0\in (0,\min_i p_i]$ for all $k\geq 0$. Then for any $y\in \R^N$, the iterates $\{x_k\}_{k\geq 1}$ of Algorithm~\ref{alg:general} satisfy: 
\begin{equation}\label{eq-unif7s}
\max\left\{ \E[F(\hat{x}_{k})] , \min_{l=1,\dots,k} \E[F(x_{l})]\right\} -F(y) \leq  
\displaystyle\frac{C}{(k-1)\theta_0+1}  \enspace , \forall k\geq 1
\end{equation}
where 
\[\hat{x}_k = \frac{x_k + \theta_0\sum_{l=1}^{k-1} x_l}{1+ (k-1)\theta_0}\]
and 
$$
C = (1-\theta_0)\left(F(x_0)-F(y)\right)+\frac{\theta_0^2}{2}\|x_0-y\|^2_{v\circ p^{-2}}.
$$
\end{theo}

\begin{theo}[ALPHA -- proximal \& accelerated] \label{th-proximal_fast}  
Let $\hat{S}$ be arbitrary proper sampling and $v\in \R^n_{++}$ be such that $(f,\hat S)\sim ESO(v)$. Choose $\theta_0\in (0,\min_i p_i]$ and define the sequence $\{\theta_k\}_{k\geq 0}$ by 
\begin{align*}\theta_{k+1}=\frac{\sqrt{\theta_k^4+4\theta_k^2}-\theta_k^2}{2}.
\end{align*} Then for any $y\in \R^N$ such that $C\geq 0$, the iterates $\{x_k\}_{k\geq 1}$ of Algorithm~\ref{alg:general} satisfy: 
\begin{equation}\label{eq-unifisu87556s}
\E[F(x_{k})]-F(y) \leq  
\displaystyle\frac{4C}{((k-1)\theta_0+2)^2} , \enspace
\end{equation}
where 
$$
C = (1-\theta_0)\left(F(x_0)-F(y)\right)+\frac{\theta_0^2}{2}\|x_0-y\|^2_{v\circ p^{-2}}.
$$
\end{theo}

In the remainder of the section we will provide the complexity analysis.

Our approach is  similar to that presented in \cite{APPROX}, but with many modifications required because we allow for an  arbitrary sampling.  We begin with some technical lemmas.

\subsection{Technical lemmas} 
 Lemma~\ref{l-xkzkc} shows that each individual block  $x^i_k$ of the variable $x_k$ is a convex 
combination of all the history blocks  $z_0^i, \dots,z_k^i$. 
Note that  due to the importance sampling, the combination coefficients $\gamma_{k,0}^{i},\dots, \gamma_{k,k}^{i}$ is now block-dependent, 
in contrast with the block-independent coefficients proved in~\cite{APPROX}.

\begin{lemma}\label{l-xkzkc}

 Let $\{x_k,z_k\}_{k\geq 0}$ be the iterates of Algorithm~\ref{alg:general}. Then for all $k\in \bN$ and $i\in [n]$
we have
\begin{align}\label{a-x_kisum}
x_k^{i}=\sum_{l=0}^k \gamma_{k,l}^{i} z_l^i,\enspace
\end{align}
where for each $i$, the coefficients $\{\gamma_{k,l}^{i}\}_{l=0,\dots,k}$  are defined recursively  by setting $\gamma_{0,0}^{i}=1$,
$\gamma_{1,0}^{i}=1-\theta_0p_i^{-1}$, $\gamma_{1,1}^{i}=\theta_0p_i^{-1}$
and for $k\geq 1$,
\begin{equation}\label{eq-gamk}
\gamma_{k+1,l}^{i}=\left\{
\begin{array}{ll}
(1-\theta_k)\gamma_{k,l}^{i} & l=0,\dots,k-1\\
 (1-\theta_k)\gamma_{k,k}^{i}+\theta_k-\theta_kp_i^{-1} & l=k\\
\theta_kp_i^{-1} & l=k+1
\end{array}
\right.
\end{equation}
so that the following identity holds,
\begin{align}\label{a-gak1ga}
\gamma_{k+1,k}^{i}+\gamma_{k+1,k+1}^{i}=(1-\theta_k)\gamma_{k,k}^{i}+\theta_k, \enspace \forall k\in \bN, \enspace i\in [n].
\end{align}
Moreover, if   $\theta_0\in (0,\min_i p_i]$ and $\{\theta_k\}_{k\geq 0}$ is a decreasing positive sequence, then for all $k\in \bN$ and  $i\in[n]$, the coefficients $\{\gamma_{k,l}^{i}\}_{l=0,\dots,k}$ are
all positive and sum to 1.
\end{lemma}

\begin{proof}Fix any $i\in[n]$.
 We proceed by induction on $k$.  It is clear from $x_0=z_0$ that
$\gamma_{0,0}^{i}=1$.
Since $x_1=y_0+\theta_0p^{-1}\cdot(z_1-z_0)$ and $y_0=x_0$, we get that 
$x_1^ i=(1-\theta_0p_i^{-1})z_0^i+\theta_0p_i^{-1} z_1^i$ 
thus $\gamma_{1,0}^{i}=1-\theta_0p_i^{-1}$ and $\gamma_{1,1}^{i}=\theta_0p_i^{-1}$.
 Assuming that~\eqref{a-x_kisum} holds for some $k\geq 1$, then
\begin{eqnarray*}
 x_{k+1}^i&=& y_k^i+ \theta_k p_i^{-1}(z_{k+1}^i-z_k^i) \;\; =\;\; (1-\theta_k)x_k^i+\theta_k z_k^i- \theta_k p_i^{-1} z_k^i+ \theta_kp_i^{-1} z_{k+1}^i\\
&\overset{\eqref{a-x_kisum}}{=}&(1-\theta_k)\sum_{l=0}^k \gamma_{k,l}^{i} z_l^i+\theta_k z_k^i- \theta_k p_i^{-1} z_k^i+ \theta_kp_i^{-1} z_{k+1}^i\\
&=&\sum_{l=0}^{k-1}(1-\theta_k) \gamma_{k,l}^{i} z_l^i+((1-\theta_k)\gamma_{k,k}^{i}+\theta_k- \theta_k p_i^{-1})z_k^i+ \theta_kp_i^{-1} z_{k+1}^i.
\end{eqnarray*}
Therefore the recursive equation~\eqref{eq-gamk} holds. 
The  identity~\eqref{a-gak1ga} can then be verified by direct substitution.
Next we assume that  $\theta_0\in (0,\min_i p_i]$ and $\{\theta_k\}_{k\geq 0}$ is a decreasing positive sequence and
 show that the linear combination in~\eqref{a-x_kisum} is a convex combination. Let $k\geq 1$.
Since $\theta_k\leq 1$, we deduce from~\eqref{eq-gamk} that
 $\{\gamma_{k+1,l}^ {i}\}_{l=0,\dots,k-1}$ are positive if $\{\gamma_{k,l}^ {i}\}_{l=0,\dots,k-1} $ are positive.  Moreover,
\begin{eqnarray*}
\gamma_{k+1,k}^{i}&\overset{\eqref{eq-gamk}}{=}&
(1-\theta_{k})\gamma_{k,k}^{i}+\theta_{k}- \theta_{k} p_i^{-1}\;\; = \;\; \theta_{k}(1- \gamma_{k,k}^{i})+\gamma_{k,k}^{i}-\theta_{k}p_i^{-1}\\
&\overset{\eqref{eq-gamk}}{=}&\theta_{k}(1- \gamma_{k,k}^{i})+(\theta_{k-1}-\theta_{k})p_i^{-1} \;\; \geq  \;\; \theta_{k}(1-\gamma_{k,k}^{i}).
\end{eqnarray*}
Then using $\theta_k \geq 0$, we conclude that 
$ \gamma_{k+1,k}^ {i} \geq 0$ if $\gamma_{k,k}^{i}\leq 1$. Besides, we have:
\begin{eqnarray*}
\sum_{l=0}^{k+1} \gamma_{k+1,l}^{i}&=&\sum_{l=0}^{k-1}\gamma_{k+1,l}^{i}+\gamma_{k+1,k}^{i}+\gamma_{k+1,k+1}^{i}\\
&\overset{\eqref{eq-gamk}}{=} & (1-\theta_k)\sum_{l=0}^{k-1}\gamma_{k,l}^{i}+(1-\theta_k)\gamma_{k,k}^{i}+\theta_k-\theta_kp_i^{-1}+\theta_kp_i^{-1} \;\; = \;\; (1-\theta_k)\sum_{l=0}^{k}\gamma_{k,l}^{i}+\theta_k.
\end{eqnarray*}
We deduce from the above facts that the coefficients $\{\gamma_{k+1,l}^{i}\}_{l=0,\dots,k+1}$ are all positive and sum to 1
if the same holds for $\{\gamma_{k,l}^{i}\}_{l=0,\dots,k}$. Since $\theta_0\leq \min_i p_i$, we know that 
 $\{\gamma_{1,0}^{i},\gamma_{1,1}^{i}\}$ are positive and sum to 1.
It follows that the same property holds for all $k\in \bN$.
\end{proof}

\begin{lemma}
 For $k\in \bN$ and $i\in[n]$, define 
\begin{align}\label{a-hatpsi}
\hat\psi^i_k \;\; \eqdef\;\;\sum_{l=0}^k \gamma_{k,l}^{i}\psi^i(z_l^i).
\end{align}
Moreover,
\begin{equation}\label{eq-psik1}
\begin{array}{ll}
\E_k[\hat\psi^i_{k+1}] \;\; = \;\; (1-\theta_k)\hat\psi_{k}^i+\theta_k\psi^i(\tilde z_{k+1}^i),\qquad \forall k \in \bN, \quad i\in [n].
\end{array}
\end{equation}
\end{lemma}
\begin{proof}
 \begin{eqnarray*}
\E_k[\hat\psi^i_{k+1}]&=&\displaystyle\sum_{l=0}^k \left[\gamma_{k+1,l}^{i} \psi^i(z_l^i)\right]+\theta_k p_i^{-1} \E_k[\psi^{i}(z_{k+1}^i)]\\
&\overset{\eqref{a-zk21}}{=}&\displaystyle\sum_{l=0}^k \left[\gamma_{k+1,l}^{i} \psi^i(z_l^i)\right]+\theta_kp_i^{-1}\big((1-p_i)\psi^i(z_k^i)+p_i\psi^i(\tilde z_{k+1}^i)\big)\\
&{=}&\displaystyle\sum_{l=0}^k \left[\gamma_{k+1,l}^{i} \psi^i(z_l^i)\right]+(p_i^{-1}-1)\theta_k\psi^i(z_k^i)+\theta_k\psi^i(\tilde z_{k+1}^i)\\
&\overset{\eqref{eq-gamk}}{=}&(1-\theta_k)\displaystyle\sum_{l=0}^{k-1} \left[\gamma_{k,l}^{i} \psi^i(z_l^i)\right]+\big(\gamma_{k+1,k}^{i}+(p_i^{-1}-1)\theta_k\big)\psi^i(z_k^i)+\theta_k\psi^i(\tilde z_{k+1}^i)\\
&\overset{\eqref{eq-gamk}}{=}&(1-\theta_k)\displaystyle\sum_{l=0}^{k-1} \left[\gamma_{k,l}^{i} \psi^i(z_l^i)\right]+\big(\gamma_{k+1,k}^{i}+\gamma_{k+1,k+1}^{i}-\theta_k\big)\psi^i(z_k^i)+\theta_k\psi^i(\tilde z_{k+1}^i)\\
&\overset{\eqref{a-gak1ga}}{=}&(1-\theta_k)\displaystyle\sum_{l=0}^{k} \left[\gamma_{k,l}^{i} \psi^i(z_l^i)\right]+\theta_k\psi^i(\tilde z_{k+1}^i)\\
&\overset{\eqref{a-hatpsi}}{=}&(1-\theta_k)\hat\psi_{k}^i+\theta_k\psi^i(\tilde z_{k+1}^i)
.
\end{eqnarray*}
\end{proof}

The next result was previously stated and used in~\cite{APPROX}.
\begin{lemma}[\cite{ChenTeboulle93,tseng2008accelerated}]\label{l-psix*}
 Let $$\xi(z)\eqdef f(y_k)+\<\nabla f(y_k),z-y_k>+\frac{\theta_k}{2}\|z-z_k\|^2_{p^{-1}\circ v},\enspace z\in \R^N.$$
Then, for $\tilde z_{k+1}$ defined in~\eqref{a-tildezk} we have:
\begin{align}\label{a-psixx}
\psi(\tilde z_{k+1})+\xi(\tilde z_{k+1})\leq \psi(y)+\xi(y)-\frac{\theta_k}{2}\|\tilde z_{k+1}-y\|^2_{p^{-1}\circ v},\enspace y\in \R^ N.
\end{align}
\end{lemma}
For a proof, see for example~\cite[Lemma 3.2]{ChenTeboulle93}.

\subsection{Recursion}
For all $k\geq 0$ define:
\begin{align}\label{a-hatpsik}
\hat \psi_k\eqdef\sum_{i=1}^n \hat \psi_k^i,\quad \hat F_k\eqdef\hat \psi_k+f(x_k).
\end{align}
 We next prove an inequality similar to the one we established in the smooth case~\eqref{eq-pr12}.

\begin{lemma} \label{lem:recursion2}  Let $\hat{S}$ be an arbitrary proper sampling and  $v\in \R^n_{++}$ be such that $(f,\hat{S})\sim ESO(v)$. Let $\{\theta_k\}_{k\geq 0}$ be arbitrary  sequence of positive numbers in $(0,1]$ and fix $y\in \R^N$. Then for the sequence of iterates produced by Algorithm~\ref{alg:general} and all $k\geq 0$, the  following recursion holds: 
\begin{equation}\label{eq-pr1233}
\E_k \left[\hat F_{k+1} + \frac{\theta_k^2}{2}\| z_{k+1}-y\|^2_{ v\circ p^{-2}}\right] \leq 
\left[ \hat F_k +
\frac{\theta_k^2}{2}\|z_k-y\|^2_{ v \circ p^{-2}}\right] - \theta_k (\hat F_k - F(y)) \enspace.
\end{equation}
\end{lemma}
\begin{proof}
If $(f,\hat S)\sim ESO(v)$, then,
 \begin{eqnarray} \notag
\E_k[f(x_{k+1})]&\overset{\eqref{a-xkplus1yk}}{=}&\E_k[f(y_k+\theta_k p^{-1}  \cdot(\tilde z_{k+1}-z_k)_{[S_k]})]\\ \notag
& \overset{\eqref{a-eso}}{\leq }&f(y_k)+\theta_k\<\nabla f(y_k),\tilde z_{k+1}-z_k>+\frac{\theta_k^2}{2}\|\tilde z_{k+1}-z_k\|^2_{p^{-1}\circ v}\\ \notag
&=&(1-\theta_k)f(y_k)-\theta_k\<\nabla f(y_k),z_k-y_k>\\ \notag
&&\qquad +\theta_k\big(f(y_k)+\<\nabla f(y_k),\tilde z_{k+1}-y_k>+\frac{\theta_k}{2}\|\tilde z_{k+1}-z_k\|^2_{p^{-1}\circ v}\big)\\ \notag
&\overset{\eqref{a-xkykzk}}{=}&(1-\theta_k)(f(y_k)+\<\nabla f(y_k),x_k-y_k>)\\
&&\qquad +\theta_k\big(f(y_k)+\<\nabla f(y_k),\tilde z_{k+1}-y_k>+\frac{\theta_k}{2}\|\tilde z_{k+1}-z_k\|^2_{p^{-1}\circ v}\big)
.\label{eq-Ekfk1}
\end{eqnarray}
We first write
\begin{eqnarray}\notag
\E_k[\hat F_{k+1}]&=&\displaystyle\E_k[\sum_{i=1}^n\hat \psi^i_{k+1}+f(x_{k+1})] \;\; \overset{\eqref{eq-psik1}}{=} \;\; \sum_{i=1}^n \left[(1-\theta_k)\hat \psi^i_{k}+\theta_k \psi^i(\tilde z_{k+1}^i)\right]+\E_k[f(x_{k+1})]
\\ 
&=&(1-\theta_k)\hat\psi_k+ \theta_k \psi(\tilde z_{k+1})+\E_k[f(x_{k+1})], \label{eq:is98g98shksss}
\end{eqnarray}
and then bound the expectation of $\hat F_{k+1}$ as follows:
\begin{eqnarray}\notag
&\qquad\E_k[\hat F_{k+1}]&\\  \notag
&\overset{\eqref{eq:is98g98shksss} 
+\eqref{eq-Ekfk1}}{\leq}& (1-\theta_k)\hat \psi_k+(1-\theta_k)(f(y_k)+\<\nabla f(y_k),x_k-y_k>)\\\notag
&&\qquad +\theta_k\big( \psi(\tilde z_{k+1})+f(y_k)+\<\nabla f(y_k),\tilde z_{k+1}-y_k>+\frac{\theta_k}{2}\|\tilde z_{k+1}-z_k\|^2_{p^{-1}\circ v}\big)\\ \notag
&\overset{\eqref{a-psixx}}{\leq} &(1-\theta_k)\hat \psi_k+(1-\theta_k)(f(y_k)+\<\nabla f(y_k),x_k-y_k>)\\\notag
&&\quad +\theta_k\big( \psi(y)+f(y_k)+\<\nabla f(y_k),y-y_k>+\frac{\theta_k}{2}\|z_k-y\|^2_{p^{-1}\circ v}-\frac{\theta_k}{2}\|\tilde z_{k+1}-y\|^2_{p^{-1}\circ v}\big)\\\notag
&\leq &(1-\theta_k)\hat \psi_k+(1-\theta_k)f(x_k)\\ \notag
&&\qquad +\theta_k\big( \psi(y)+f(y)+\frac{\theta_k}{2}\|z_k-y\|^2_{p^{-1}\circ v}-\frac{\theta_k}{2}\|\tilde z_{k+1}-y\|^2_{p^{-1}\circ v}\big)\\ \notag
&=&(1-\theta_k)\hat F_k+\theta_k F(y)+\frac{\theta_k^2}{2}\left(\|z_k-y\|^2_{p^{-1}\circ v}-\|\tilde z_{k+1}-y\|^2_{p^{-1}\circ v}\right)\\
&\overset{\eqref{a-zk21}}{=}&(1-\theta_k)\hat F_k+\theta_k F(y)+\frac{\theta_k^2}{2}\E_k\left[\|z_k-y\|^2_{p^{-2}\circ v}-\|y- z_{k+1}\|^2_{p^{-2}\circ v}\right]
,\enspace \forall y\in \R^ N. \notag
\end{eqnarray}
Therefore, for all $y\in \R^N$,
\[
\E_k\left[ \hat F_{k+1}-F(y)+\frac{\theta_k^2}{2}\|y- z_{k+1}\|^2_{p^{-2}\circ v}\right] 
\leq (1-\theta_k)( \hat F_k-F(y))+\frac{\theta_k^2}{2}\|z_k-y\|^2_{p^{-2}\circ v}\enspace.\qedhere
\]
\end{proof}

\subsection{Proof of Theorems~\ref{th-proximal_normal} and~\ref{th-proximal_fast}}\label{sec-proof}

Using the same reasoning  as that in the proof of Theorems~\ref{th-smooth_normal} and \ref{th-smooth_fast},  we analyze recursion of Lemma~\ref{lem:recursion2} and obtain:
\begin{itemize}
 \item If $\theta_k=\theta_0 \in (0,1]$ for all $k\geq 0$, then
$$
 \frac{\E\left[  \hat F_k-F(y)+\theta_0\sum_{l=1}^{k-1} (\hat F_{l}-F(y))\right]}{1+(k-1)\theta_0}  \leq \frac{(1-\theta_0)(\hat F_0-F(y))+\frac{\theta_0^2}{2}\|x_0-y\|_{v\circ p^{-2}}}{1+\theta_0(k-1) },\enspace \forall k\geq 1.
$$

\item If $\theta_{k+1}=\frac{\sqrt{\theta_k^4+4\theta_k^2}-\theta_k^2}{2}$ for all $k\geq 0$, $\theta_0\in (0,1]$ and $\hat F_0 \leq F(y)$, then
\begin{align}\label{a-EhatFkF*}
\E \left[\hat F_{k+1}-F(y)\right] \leq \frac{4\left((1-\theta_0)(\hat F_0-F(y))+\frac{\theta_0^2}{2}\|x_0-y\|_{v\circ p^{-2}} \right)}{(\theta_0(k-1)+2)^2},\enspace \forall k\geq 1.
\end{align}
\end{itemize}
Finally, by Lemma~\ref{l-xkzkc}, for the latter two choices of $\{\theta_k\}$, if in addition $\theta_0\in (0,\min_i p_i]$, then for $k\geq 1$, each block of the vector $x_k$ is a convex combination of the corresponding blocks of the vectors $z_0,\dots,z_k$. By the convexity of
each function $\psi^i$, we get:
\begin{align}\label{a-psiwksum}
\psi(x_k)=
\sum_{i=1}^n\psi^i(x_k^i)=\sum_{i=1}^n\psi^i(\sum_{l=0}^k\gamma_{k,l}^{i}z_l^i)\leq \sum_{i=1}^n \sum_{l=0}^k \gamma_{k,l}^{i}\psi^i(z_l^i)=\hat\psi_k.
\end{align}
Hence, Theorem~\ref{th-proximal_normal} and~\ref{th-proximal_fast} hold  by the fact that $F(x_k)\leq \hat F_{k} $ for all $k\in \bN$ and $\hat F_0=F(x_0)$. Note that the condition $\theta_0\in (0,\min_i p_i]$ is only needed to prove~\eqref{a-psiwksum}. Thus it can be relaxed to $\theta_0\in (0,1]$ if $\psi \equiv 0$, in which case $\hat F_k= F(x_k)$ and Theorem~\ref{th-smooth_normal} and~\ref{th-smooth_fast} follows.

\section{Conclusion}\label{sec:conclusion}

In this paper we propose a general randomized coordinate descent method which can be specialized to serial or parallel and accelerated or non-accelerated variants, with or without importance sampling. Based on the technical assumption which captures 
in a compact way certain smoothness properties of the function in a random 
subspace spanned by the sampled coordinates, we provide a unified complexity analysis which allows to derive as direct corollary the convergence results for the multiple variants of the general algorithm. We focused on the minimization of non-strongly convex function. Further study on a unified algorithm and complexity analysis for both strongly and non-strongly convex objective functions can be investigated.

{
\bibliographystyle{plain}
\bibliography{biblio}
}

\appendix
\clearpage

\section{Special cases of ALPHA in the proximal setup}

Extending the discussion presented in Section~\ref{sec-specialcasesm} which focused on the smooth case, we now present four  special cases of Algorithm~\ref{alg:general} for solving problem~\eqref{eq-probse}.  

\subsection{Special case 1: proximal gradient descent}
Specializing Algorithm~\ref{alg:general} to
 $\hat S=[n]$  and $\theta_k= 1$ for all $k\geq 1$, we obtain the classical proximal gradient descent algorithm.  Note that in this case the three sequences $\{x_k,y_k,z_k\}_{k\geq 0}$ reduce to one sequence.
\begin{algorithm}[H]
\begin{algorithmic}[1]
\STATE \textbf{Parameters:} vector $v\in \R_{++}^ n$ such that~\eqref{a-esofull1} holds
\STATE \textbf{Initialization:} choose $x_0\in \dom \psi$
\FOR{$k\geq 0$} 
\FOR{ $i\in [n]$} 
\STATE
$
x_{k+1}^i=\argmin_{x\in \R^{N_i}} \big\{\< \nabla_i f(x_k), x>+\frac{v_i }{2 } \|x-x_k^i\|_i^2+\psi^i(x) \big\}
$ \\
\ENDFOR\\
\ENDFOR
\end{algorithmic}
\caption{Proximal Gradient Descent for solving~\ref{eq-probse}}
\label{alg:pgd}
\end{algorithm}
\begin{coro}\label{coro-pgd}
 For any optimal solution $x_*$ of~\eqref{eq-probse}, the output of Algorithm~\ref{alg:pgd} for 
all $k\geq 1$ satisfies: 
\begin{align}\label{a-pfxkfstarleq}
F(x_{k})-F(x_*)\leq  
\displaystyle\frac{\|x_0-x_*\|^2_{v}}{2k}.
\end{align}
In particular, for $\epsilon>0 $, if 
\begin{align}\label{a-coro-pgd}
k\geq \frac{\|x_0-x_*\|^2_{ v}}{2\epsilon},
\end{align}
then $
F(x_k)-F(x_*)\leq \epsilon$.
\end{coro}
The proof follows as a corollary of Theorem~\ref{th-proximal_normal}, with additional remark~\eqref{a-xyz} which holds in this special case. The reader can refer to the proof of Corollary~\ref{coro-gd}.

Corollary~\ref{coro-pgd} can be found in classical textbooks on convex optimization, see for example~\cite{Nesterov:2010RCDM}.

\subsection{Special case 2: accelerated proximal gradient descent}

By choosing $\hat S=[n]$ (with probability 1), $\theta_0=1$ and the sequence $\{\theta_k\}_{k\geq 0}$ according to~\eqref{a-thetak+1}, ALPHA (Algorithm~\ref{alg:general}) reduces to  accelerated proximal gradient descent.

\begin{algorithm}[H]
\begin{algorithmic}[1]
\STATE \textbf{Parameters:} vector $v\in \R_{++}^ n$ such that~\eqref{a-esofull1} holds
\STATE \textbf{Initialization:} {choose $x_0\in \dom(\psi)$, set $z_0=x_0$ and $\theta_0=1$}
\FOR{$k\geq 0$} 
\FOR{ $i\in [n]$} 
\STATE
$
z_{k+1}^i=\argmin_{z\in \R^{N_i}} \big\{\< \nabla_i f((1-\theta_k)x_k+\theta_kz_k), z>+\frac{\theta_k  v_i }{2p_i} \|z-z_k^i\|_i^2+\psi^i(z) \big\}
$ \\
\ENDFOR\\
\STATE $x_{k+1}=(1-\theta_k)x_k+\theta_k z_{k+1}$\\
\STATE $\theta_{k+1}=\frac{\sqrt{\theta_k^4+4\theta_k^2}-\theta_k^2}{2}$
\ENDFOR
\end{algorithmic}
\caption{Accelerated proximal gradient descent~\cite{AuslenderTeboulle06}}
\label{alg:apgd}
\end{algorithm}

\begin{coro}\label{coro-apgd}
 For any optimal solution $x_*$ of~\eqref{eq-probse}, the output of Algorithm~\ref{alg:apgd} for 
all $k\geq 1$ satisfies: 
$$
F(x_k)-F(x_*)\leq  \frac{2\|x_0-x_*\|^2_{ v}}{(k+1)^2}.\quad 
$$
In particular, for $0<\epsilon<  \frac{1}{2}\|x_0-x_*\|^2_{ v} $, if 
\begin{align}\label{a-coro-apgd}
k\geq \sqrt{\frac{2\|x_0-x_*\|^2_{ v}}{\epsilon}}-1,
\end{align}
then $ F(x_k)-F(x_*)\leq \epsilon$.
\end{coro}
As discussed for the unconstrained case in Section~\ref{subsec-agd}, Algorithm~\ref{alg:apgd} and Corollary~\ref{coro-apgd}  can be attributed to Tseng~\cite{tseng2008accelerated}.

\subsection{Special case 3: Parallel Coordinate Descent Method (PCDM)}\label{subsec-PCDM}

Le $\hat S$ be a proper uniform  sampling (i.e., sampling for which $p_i = \Prob(i\in \hat{S})>0$ is the same for all $i\in[n]$, necessarily equal to $\Exp[|\hat{S}|]/n$). Furthermore, letting 
$\tau=\E[|\hat S|]$, choose $\theta_k=\tau/n$ for all $k$. For this choice of the parameters  \eqref{a-xyz} holds,  ad hence for Algorithm~\ref{alg:general} reduces to the  PCDM method of Richt\'{a}rik and Tak\'{a}\v{c}~\cite{PCDM}.

\begin{algorithm}[H]
\begin{algorithmic}[1]
\STATE \textbf{Parameters:}   uniform proper random sampling $\hat S$,  positive vector $v\in \R_{++}^ n$ such that $(f,\hat S)\sim ESO(v)$
\STATE \textbf{Initialization:} {choose $x_0\in \dom \psi$ and set $\tau=\Exp[|\hat S|]$}
\FOR{$k\geq 0$} 
\STATE Generate $S_k\sim \hat S $\\
\STATE $x_{k+1}\leftarrow x_k$
\FOR{ $i\in S_k$} 
\STATE
$
x_{k+1}^i=\argmin_{x\in \R^{N_i}} \big\{\< \nabla_i f(x_k), x>+\frac{v_i n}{2 \tau } \|x-x_k^i\|_i^2+\psi^i(x) \big\}
$ \\
\ENDFOR\\
\ENDFOR
\end{algorithmic}
\caption{Parallel Coordinate Descent Method (PCDM)~\cite{PCDM}}
\label{alg:PCDM}
\end{algorithm}
 By applying Theorem~\ref{th-proximal_normal} and using the same reasoning as in the proof of Corollary~\ref{coro-gd} we obtain the following {\em new} convergence result for PCDM.

\begin{coro}\label{coro-PCDM}
For any optimal solution $x_*$ of~\eqref{eq-probse}, the output of Algorithm~\ref{alg:PCDM} for 
all $k\geq 1$ satisfies: 
\begin{align}\label{a-coro-PCDM}
\E[F(x_{k})]-F(x_*)\leq  
\displaystyle\frac{n}{(k-1)\tau+n}\left[
\left(1-\frac{\tau}{n}\right)\left(F(x_0)-F(x_*)\right)+\frac{1}{2}\|x_0-x_*\|^2_{v}\right].
\end{align}
In particular, for $0<\epsilon<  \left(1-\tau/n\right)\left(F(x_0)-F(x_*)\right)+\frac{1}{2}\|x_0-x_*\|^2_{ v} $, if 
\begin{align*}
k\geq \frac{\displaystyle \left(n-\displaystyle \tau\right)\left(F(x_0)-F(x_*)\right)+\displaystyle\frac{n}{2}\|x_0-x_*\|^2_{ v}}{\displaystyle \tau \epsilon }- \frac{n}{\displaystyle\tau}+1,
\end{align*}
then $\E[F(x_k)-F(x_*)]\leq \epsilon$.
\end{coro}
A high-probability result involving the level-set distance
\begin{align}\label{a-levelsetdis}
\cR_{v}(x_0,x_*)
\eqdef \max_x \{ \|x-x_*\|_v^2: F(x) \leq F(x_0)\}<+\infty
\end{align}
was  provided in~\cite{PCDM} for
PCDM (Algorithm~\ref{alg:PCDM}).  
Although not explicitly stated in the paper, it is apparent from the proof that their approach yields the following  rate:
\begin{align}\label{a-PCDM-previous}
\E[F(x_k)]-F(x_*) \leq \frac{2n\max \{ \cR_{v}(x_0,x_*), F(x_0)-F(x_*)
\} }{2n\max \{ \cR_{v}(x_0,x_*)/\left(F(x_k-F(x_*)\right), 1
\}+{\tau} k}
\end{align}
Since $\cR_{v}(x_0,x_*) \geq \|x_0-x_*\|_v^2 $, it is clear that for sufficiently large $k$, our rate~\eqref{a-coro-PCDM} is better than~\eqref{a-PCDM-previous} .

\subsection{Special case 4: APPROX with importance sampling}
Finally, let $\{\theta_k\}_{k\geq 0}$ be chosen in accordance with~\eqref{a-thetak+1}. In this case, ALPHA (Algorithm~\ref{alg:general}) reduces to an accelerated coordinate descent method with arbitrary sampling $\hat S$ for solving the proximal minimization problem~\ref{eq-probse}.

\begin{algorithm}[H]
\begin{algorithmic}[1]
\STATE \textbf{Parameters:} proper random sampling $\hat S$ with probability vector $p=(p_1,\dots,p_n)$, $v\in \R_{++}^ n$ such that $(f,\hat S)\sim ESO(v)$
\STATE \textbf{Initialization:} {choose $x_0\in \dom(\psi)$, set   $z_0=x_0$ and $\theta_0=\min_i p_i$}
\FOR{$k\geq 0$} 
\STATE $y_k=(1-\theta_k)x_k+\theta_kz_k$ \\
\STATE Generate $S_k\sim \hat S $\\
\STATE $z_{k+1}\leftarrow z_k$
\FOR{ $i\in S_k$} 
\STATE
$
z_{k+1}^i=\argmin_{z\in \R^{N_i}} \big\{\< \nabla_i f(y_k), z>+\frac{\theta_k  v_i }{2p_i} \|z-z_k^i\|_i^2+\psi^i(z) \big\}
$ \\
\ENDFOR\\
\STATE $x_{k+1}=y_k+\theta_k p^{-1}\cdot(z_{k+1}-z_k)$\\
\STATE $\theta_{k+1}=\frac{\sqrt{\theta_k^4+4\theta_k^2}-\theta_k^2}{2}$
\ENDFOR
\end{algorithmic}
\caption{APPROXis (APPROX~\cite{APPROX} with importance sampling)}
\label{alg:approxn1}
\end{algorithm}

APPROXis is a generalization of APPROX~\cite{APPROX} from a uniform sampling to an arbitrary sampling: we recover APPROX if 
 $\hat S$ is uniform with $\tau=\Exp[|\hat S|]$ and  $\theta_0=\tau/n$.
\begin{coro}\label{coro-APPROX}
 For any optimal solution $x_*$ of~\eqref{eq-probse}, the output of Algorithm~\ref{alg:approxn1} for 
all $k\geq 1$ satisfies: 
$$
\E[F(x_k)-F(x_*)]\leq  \frac{4\left[ \left(1-\displaystyle\min_i p_i\right)\left(F(x_0)-F(x_*)\right)+\displaystyle\min_i\frac{ p_i^2}{2}\|x_0-x_*\|^2_{v\circ p^{-2}}\right]}{((k-1) \displaystyle\min_i p_i+2)^2} .\quad 
$$
In particular, for $\epsilon>0$, if 
\begin{align}\label{a-coro-APPROX}
k\geq \frac{\displaystyle {2} \sqrt{\left(1-\displaystyle\min_i p_i\right)\left(F(x_0)-F(x_*)\right)+\displaystyle\min_i\frac{ p_i^2}{2}\|x_0-x_*\|^2_{v\circ p^{-2}}}}{\displaystyle\min_i p_i\sqrt{\epsilon}}- \frac{2}{\displaystyle\min_i p_i}+1,
\end{align}
then
$\E[F(x_k)-F(x_*)]\leq \epsilon$.
\end{coro}
\begin{proof}
This is a direct corollary of Theorem~\ref{th-proximal_fast} by taking $\theta_0=\min_i p_i$.
\end{proof}
 If $\hat S$ is a uniform sampling with $\tau=\Exp[|\hat S|]$ and $\theta_0=\tau/n$, then we recover the
 convergence result established for APPROX~\cite[Theorem 3]{APPROX}, as well as all the special cases that APPROX can recover, including the fast distributed  coordinate descent method Hydra$^2$~\cite{Hydra2}.

\end{document}